\documentclass[a4paper,11pt]{scrartcl}

\RequirePackage{xr}
\RequirePackage[OT1]{fontenc}
\RequirePackage{amsthm,amsmath,amssymb,dsfont,bm}
\RequirePackage{algorithm}
\RequirePackage{algpseudocode}
\RequirePackage{graphicx} 
\RequirePackage[section]{placeins}
\RequirePackage{color}
\RequirePackage{chngcntr}
\RequirePackage{enumerate}
\usepackage[shortlabels]{enumitem}
\RequirePackage{longtable}
\usepackage{multirow}
\usepackage{mathtools}
\usepackage[english]{babel} 
\usepackage{xcolor}

\RequirePackage[OT1]{fontenc} \RequirePackage{amsthm}
\RequirePackage{amsfonts} \RequirePackage{amssymb}
\RequirePackage{bbm}

\usepackage[normalem]{ulem}

\RequirePackage{graphicx}
\RequirePackage{verbatim}
\usepackage{url}

\numberwithin{equation}{section}
\RequirePackage{a4wide} 

\theoremstyle{plain}

\newtheorem{theorem}{Theorem}[section]

\newtheorem{remark}{Remark}[section]

\def\@bysame#1{\vrule height 1.5pt depth -1pt width 3em \hskip
0.5em\relax}

\newcommand{\N}{ \mathbb{N} }

\newcommand{\R}{ \mathbb{R} }

\newcommand{\trunc}[1]{ {\lfloor #1 \rfloor} }
\newcommand{\wh}[1]{ \widehat{ #1 } }
\newcommand{\wt}[1]{ \widetilde{ #1 } }

\newcommand{\eins}{{\bm 1}}

\newcommand{\vecM}{{\bm M}}

\newcommand{\vecw}{{\bm w}}

\newcommand{\vecX}{{\bm X}}

\newcommand{\bfxi}{\bm \xi}
\newcommand{\bfmu}{\bm \mu}

\newcommand{\bfSigma}{\bm\Sigma}

\newcommand{\Var}{{\mbox{Var\,}}}

\usepackage{subfigure}
\usepackage{mathtools} 
\usepackage{bm}
\usepackage{dsfont}
\usepackage{color}
\usepackage{amsthm}

\usepackage{float}                 

\begin{document}

\begin{center}
	\begin{minipage}{.8\textwidth}
		\centering 
		\LARGE High-Confident Nonparametric Fixed-Width Uncertainty Intervals and Applications to Projected High-Dimensional Data and Common Mean Estimation\\[0.5cm]
		
		\normalsize
		\textsc{Yuan-Tsung Chang 
			and Ansgar Steland}\\[0.1cm]

		Department of Social Information\\ 
Mejiro University \\ 
4-31-1 Nakaochiai, Shinjuku--ku \\ 
Tokyo 161--8539, Japan \\ {\em and} \\

		Insitute of Statistics \\
		RWTH Aachen University\\
		Aachen, Germany\\
		Email: \verb+steland@stochastik.rwth-aachen.de+
		
	\end{minipage}
\end{center}

\begin{abstract}
Nonparametric two-stage procedures to construct fixed-width confidence intervals are studied to quantify uncertainty. It is shown that the validity of the random central limit theorem (RCLT) accompanied by a consistent and asymptotically unbiased estimator of the asymptotic variance already guarantees consistency and first as well as second order efficiency of the two-stage procedures. This holds under the common asymptotics where the length of the confidence interval tends to $0$ as well as under the novel proposed high-confident asymptotics where the confidence level tends to $1$. The approach is motivated by and applicable to data analysis from distributed big data with non-negligible costs of data queries. The following problems are discussed: Fixed-width intervals for a the mean, for a projection when observing high-dimensional data and for the common mean when using nonlinear common mean estimators under order constraints. The procedures are investigated by simulations and illustrated by a real data analysis.
\end{abstract}

\textit{Keywords:} Big data, data science, high-dimensional data, jackknife, sequential analysis, sequential sampling 

\section{Introduction}

In this paper, we study fully nonparametric two-stage procedures to construct a fixed-width  interval for a parameter to quantify uncertainty. Both the common high-accuracy framework, where the asymptotics assumes that the width of the interval shrinks, and a novel high-confident framework is studied. Under {\em high-confident} asymptotics the required uncertainty in terms of the width of the interval is fixed and the asymptotics assumes that the confidence level increases. General sufficient conditions are derived which yield consistency and efficiency for both frameworks. We study three statistical problems: Nonparametric fixed-width intervals for the mean of univariate data which may be the most common setting, for the mean projection of high-dimensional data to illustrate the application to big data, and for the common mean of two samples as a classic statistical problem leading to a {\em nonlinear} estimator, which has not yet treated  in the literature from a two-stage sampling perspective. 
The focus is on two-stage procedures, as they provide a good compromise between the conflicting goals of a minimal sample size, which requires purely sequential sampling, and feasibility in applications in terms of required computing ressources and logistic simplicity, which is better matched by one- or two-stage sampling procedures. 

Two-stage sequential sampling is a well established approach motivated by the aim to make statistical statements with minimal samples sizes without relying on purely sequential sampling. Instead, the data is sampled in two batches, a first-stage sample and a second-stage sample if required. At the second stage the final sample size is determined using the information contained in the first-stage (pilot) sample. The development of such procedures was mainly motivated by the need to base statistical inference on small samples in a world where large samples are not available. But this technique is also of interest in various areas including emerging ones such as data science and big data, where massive amounts of variables are collected and need to be processed and analyzed: When analyzing big data distributed over many nodes of a network each single query may be associated with a high response time and substantial data transmission costs ruling out a purely sequential sampling strategy, since the benefit of less required observations on average is overcompensated by the high costs for each query. Contrary, the two-stage methods proposed in this paper allow to estimate efficiently means of the variables and their projections with preassigned accuracy and confidence. The general construction of the sample size rules mainly follows the established approach. But, compared to the existing literature, we use a slightly modified first-stage sample size rule which takes into account prior knowledge and historical estimates, respectively, of the data uncertainty. Our studies indicate that even if we use only three data points to get a rough guess of variability, the resulting first-stage sample sizes comes much closer to the actually required sample size, thus avoiding oversampling at this stage. In the context of distributed data, the proposed methods with this {\em three-observations-rule} need at most three database queries. 

This paper contributes to the existing literature on two-stage procedures, see \cite{Stein1945}, \cite{Muk1980}, \cite{GoshMukhoSen1997}, \cite{MukSilvaBook2009} and the references given therein, by proposing a concrete nonparametric procedure with the following properties: For {\em any} estimator of the mean (or a parameter $ \theta $), which satisfies the random central limit theorem and whose asymptotic variance can be estimated by a consistent and asymptotically unbiased estimator, the random sample size leading to the proposed fixed-width confidence interval is consistent and asymptotically unbiased for the optimal sample size. Further, the procedure yields the right asymptotic coverage probability and exhibits first order as well as second order efficiency. 

Further, and more importantly, we go beyond the classic framework which establishes the above properties when the width of the confidence interval tends to $0$. We argue that this is to some extent counter-intuitive in view of the posed problem to construct a {\em fixed-width} confidence interval. It also limits the approximate validity of the results to cases where one aims at high-accuracy estimation. But in many applications it is more appropriate to fix the width of the confidence interval and to require that a larger number of observations is due to a higher confidence level. Therefore, we propose a novel framework and study the construction of a fixed-width interval under {\em high-confidence} asymptotics (equivalently: {\em low-error-probability} asymptotics). This is also motivated by the fact that in many areas such as high-quality high-throughput production engineering or statistical genetics and brain research, where high-dimensional data is collected, large confidence levels and small significance levels, respectively, are in order and used in practice. For example, in production engineering the accuracy is fixed by the technical specifications and not by the statistician, and in genetics as well as in brain research small error probabilities are required to reach scientific relevance and to take multiple testing into account. 

It is shown that the proposed two-stage procedure is valid under high-confident asymptotics and exhibits first and second order efficiency properties, as long as the parameter estimator satisfies the random central limit theorem and a consistent and asymptotically unbiased estimator of the asymptotic variance is at our disposal. 

Having in mind big data sets with a large number of variables, we then apply the general results to projections of  high-dimensional data. It is  assumed that the observations are given by a data stream of (possibly) increasing dimension, which is sampled in batches by our two-stage procedure. Two-stage procedures for high-dimensional data has been studied in-depth in \cite{AoshimaYata2011} assuming that the dimension, $p$, tends to $\infty $ and the sample size is either fixed or tends to $ \infty$ as well. Here we consider a projection of high-dimensional data, where, when having sampled $n$ observations, the projection may depend on the sample size $n$.  The asymptotic properties (consistency and efficiency) of the fixed-width confidence interval for the mean projection hold for high-accuracy asymptotics  as well as high-confident asymptotics. The dimension $p$ may be increasing with $n$ in an unconstrained way. 

As an interesting and non-trivial classical application, we consider the problem of common mean estimation. Here one aims at estimating the mean from two samples assuming that they have the same mean but possibly different or ordered variances. Many of the estimators proposed and studied in the literature are given by a convex combination of the sample means with convex weights depending on the sample means and the sample variances. 

The paper is organized as follows. Section~\ref{SectionFWI} studies nonparametric two-stage fixed-width confidence intervals for the mean under both asymptotic frameworks, starting with the usual high-accuracy approach and then discussing the novel  high-confident asymptotics. Section~\ref{Sec: HDData} provides the results when dealing with a projection of high-dimensional data. Common mean estimation is treated in Section~\ref{Sec: CommonMean}. Results from simulations and a data example are provided in Section~\ref{Sec: Simulations}.
 
\section{Nonparametric two-stage fixed-width confidence interval}
\label{SectionFWI}

Let $ Y_1, Y_2, \dots $ be i.i.d.($F$) observations with mean $ \mu $ and finite variance $ \sigma^2 \in (0, \infty) $. Further, let $ \wh{\mu}_n $ be an estimator for $ \mu $ using the first $n$ observations $ Y_1, \dots, Y_n $. We focus on the mean as the parameter of interest, but it is easy to see that all results remain true for any univariate parameter $ \theta = \theta(F) $ and an estimator $ \wh{\theta}_n $. 

The classical approach to the construction of a confidence interval is based on a sample of fixed (but large) sample size $N$ and determines a random interval $ [U_N, V_N] $ depending on the sample(s), such that its coverage probabiliy equals the given confidence level $ 1-\alpha \in (0,1) $ for each $N \ge 1 $, or has asymptotic coverage $ 1-\alpha $, as $ N $ tends to $ \infty $. As a consequence, the length $L = V_N - U_N $ of the interval, which represents the reported accuracy, is random.

There are, however, situations where we want to report an interval of a fixed, preassigned accuracy $d$, symmetric around the point estimator of $ \mu $, so that $ L = 2d $. Then the coverage probability of the resulting interval $ [\wh{\mu}_N - d, \wh{\mu}_N + d ] $ depends on the distribution of $ \wh{\mu}_N $ and the sample size $N$ becomes the parameter we may select to achieve a certain confidence level.  In mathematical terms, we wish to find some $N$, so that the interval around the  estimator $ \wh{\mu}_N $ based on two samples of size $N$ has coverage probability
\begin{equation}
	\label{AsCoverage}
	P \left( [ \wh{\mu}_N - d, \wh{\mu}_N  + d ] \ni \mu \right) = 1-\alpha + o(1),
\end{equation}
as the precision parameter $d$ tends to $0$. The $ o(1) $ term is required as the CLT resp. RCLT for $ \wh{\mu}_N $ is used to construct a solution. 

Usually, $ d $ is small and $N$ increases  when $d$ decreases. Thus, it is reasonable to consider asymptotic properties as $d$ tends to $0$. We shall, however, also consider the case of a fixed accuracy $d$, not necessarily 'small', and study asymptotic properties when the confidence level tends to $1$.

Suppose the estimator $ \wh{\mu}_n $ satisfies the central limit theorem, i.e.
\begin{equation}
	\label{AssCLT}
	\sqrt{n}( \wh{\mu}_n - \mu) \stackrel{d}{\to} N(0, \sigma_\mu^2), 
\end{equation}
as $ n \to \infty $, for some positive constant $ \sigma_\mu^2 $, the asymptotic variance of our estimator for $ \mu $. Throughout the paper we shall assume that we have an estimator for $ \sigma_\mu^2 $ at our disposal, which we denote by $  \wh{\sigma}_{n}^2 = \wh{\sigma}_{\mu,n}^2 $ if it is based on the first $n$ observations. The most common choice for $ \wh{\mu}_n $ is, of course, the sample average $ \overline{Y}_n = \frac{1}{n} \sum_{i=1}^n Y_i $, which satisfies (\ref{AssCLT}) with $ \sigma_\mu^2 = \sigma^2 $. The canonical estimator for $ \sigma_\mu^2 $ is $ S_n^2 = \frac{1}{n-1} \sum_{i=1}^n (Y_i - \overline{Y}_n)^2 $. When considering a parameter $ \theta $ estimated by $ \wh{\theta}_n $ such that the analog of (\ref{AssCLT}) holds, i.e., $ \sqrt{n}( \wh{\theta}_n - \theta ) \stackrel{d}{\to} N(0, \sigma_\theta^2) $, one formally replaces $ \sigma_\mu^2 $ by $ \sigma_\theta^2 $ and needs an estimator $ \wh{\sigma}_{\theta,n}^2 $ having the properties required in Assumption (E) in the next section. For simplicity of presentation and proofs, we stick to the case of the mean, however.

 Invoking the CLT for $ \wh{\mu}_N $, it is easy to see that the problem is solved by the {\em asymptotically optimal} sample size $ N_{opt} = \lceil N_{opt}^* \rceil $, where
\begin{equation}
	\label{Formula_N_opt_star}
	N_{opt}^* = \frac{ \sigma_{\mu}^2 \Phi^{-1}(1- \alpha/2)^2 }{d^2},
\end{equation}
as the left hand side of (\ref{AsCoverage}) is equal to $ P( | \sqrt{n}( \wh{\mu}_N - \mu ) | \le d ) = 2\Phi( \sqrt{N} d/\sigma_\mu ) - 1 + o(1) $. Observe that $ N_{opt}^* \to \infty $, if $ d \downarrow 0 $, which in turn justifies the application of the CLT. 

If $ \sigma_\mu^2 $ were known, then $ \lceil N_{opt}^* \rceil $ would solve the posed problem. The proposed two-stage procedures draws a random sample of size $ N_0 $ at the first stage, which is larger or equal to a given minimal sample size $ \overline{N}_0 $. The first stage sample size $ N_0 $ will be larger if the required precision gets smaller. The first-stage sample is used to estimate the uncertainty of the estimator, and that (random) estimate is then used to specifiy the final sample size $ \wh{N}_{opt}^* $ used at the second stage. Before discussing how one should specify $ N_0 $ and $ \wh{N}_{opt}^* $, let us summarize the basic algorithm:

\noindent
\textbf{Preparations:} Specify the minimal sample size $ \overline{N}_0 $, the confidence level $ 1-\alpha $ and the precision $d$.

\noindent
\textbf{Stage I:} Draw an initial sample of size $ N_0 \ge \overline{N}_0 $, in order to estimate unknowns (in our case $ \sigma_\mu^2 $) based on that data, yielding an estimator $ \wh{N}_{opt}^* $  for $ N_{opt}^* $, i.e. a random sample size.

\noindent
\textbf{Stage II:} Draw additional $ \wh{N}_{opt}^* - N_0 $ observations to obtain a sample of size $ \wh{N}_{opt}^* $. Estimate $ \mu $ by $ \wh{\mu}_{\wh{N}_{opt}^*} $ and 

\noindent
\textbf{Solution:} output the fixed-width confidence interval $ \left[\wh{\mu}_{\wh{N}_{opt}^*}  - d, \wh{\mu}_{\wh{N}_{opt}^*} + d \right] $.

\subsection{Fixed-width interval under high-accuracy asymptotics}

Let us first study the classical approach to fix the confidence level $ 1- \alpha $ and to assume that the accuracy is small suggesting to investigate approximations for $ d \downarrow 0$. This framework can be called {\em high-accuracy asymptotics}.

In the sequel, we review part of the literature which focused on normal data and the associated optimal estimators. We follow the arguments developed for the Gaussian case to motivate our fully nonparametric proposal where $ \wh{\mu}_n $ may be an arbitrary estimator satisfying the required regularity assumptions stated below in detail.

The original Stein procedure, see  \cite{Stein1945},  addresses Gaussian i.i.d. observations and estimates $ \mu $ by the sample mean, such that $ \sigma_\mu^2 = \sigma^2 $ and a natural estimator for $ \sigma_\mu^2 $ based on $ Y_1, \dots, Y_n $ is $ S_n^2 $. Stein uses the rule
\[
N = \max\left\{ \overline{N}_0, \biggl\lfloor  \frac{t(\overline{N}_0-1)_{1-\alpha/2}^2 S_{\overline{N}_0}^2 }{ d^2 } \biggr\rfloor+ 1 \right\}.
\]
As  $ \overline{N}_0 $ is fixed, the procedure turns out to be inconsistent. To overcome this issue, \cite{ChowRobbins1965} proposed a purely sequential rule, namely
\[
N = \inf\{ n \ge \overline{N}_0 : n \ge t(n-1)_{1 - \alpha/2}^2 S_n^2 / d^2 \}.
\]
\cite{Muk1980} noted that one gets for small $d$ the lower bound $ N \ge  \Phi^{-1}(1-\alpha/2) $ and 
proposed to increase the variance estimate slightly by $ 1/n $. Indeed, for small enough $d $ we have the lower bound $ N \ge t(\overline{N}_0-1)_{1-\alpha/2} / d $, if one replaces the estimate $  S_{n_0}^2 $ by $ S_{n_0}^2 + 1/n $, since for $ d \le 1 $ 
\begin{align*}
	\frac{ t(n-1)_{1 - \alpha/2}^2 (S_n^2 + 1/n)}{d^2} 
	& = \frac{ t(n-1)_{1 - \alpha/2}^2 S_n^2 }{d^2} + \frac{ t(n-1)_{1 - \alpha/2}^2 }{ nd^2}  \\
	& \ge \frac{ t(n-1)_{1 - \alpha/2}^2 } { nd^2 },
\end{align*} 
such that any $ n \ge \overline{N}_0 $ with $ n \ge  t(n-1)_{1 - \alpha/2}^2 (S_n^2 + 1/n) / d^2 $ satisfies $ n \ge t(n-1)^2_{1-\alpha/2} / (nd^2) $ and hence $ n^2  \ge n \ge t(n-1)^2_{1-\alpha/2} / d^2 $. This leads to $ n \ge t(N_0-1)_{1-\alpha/2} / d $ for any such $n$, such that we obtain the lower bound $ N \ge t(\overline{N}_0 - 1)_{1-\alpha/2} / d $. 
Therefore, the purely sequential rule
\[
	N' = \inf\{ n \ge \overline{N}_0 : n \ge t(n-1)_{1 - \alpha/2}^2 (S_n^2+1/n) / d^2 \}
\]
satisfies $ N' \ge \max\{ \overline{N}_0,  t(\overline{N}_0-1)_{1-\alpha/2} / d \} $. 

The idea is to now to use this lower bound
$
	\max\{ \overline{N}_0, \lfloor t(\overline{N}_0-1)_{1-\alpha/2} / d \rfloor + 1 \},
$
 as the first-stage sample size for Gaussian data; for non-normal samples one replaces $ t(\overline{N}_0-1)_{1-\alpha/2} $ by the corresponding quantile, $ \Phi^{-1}(1-\alpha/2) $, of the standard normal distribution. However, this rule does not take into account the scale of data and can lead to unrealistically large sample sizes, see the data example in Section~\ref{Sec: Simulations}. \cite{Muk1980} has proposed the modified rule $
 \max\{ \overline{N}_0, \lfloor ( t(\overline{N}_0-1)_{1-\alpha/2} / d )^{2/(1+\gamma)} \rfloor +1 \}, $ for some $ 0 < \gamma < \infty $. $ \gamma $ can be selected to obtain a reasonable first-stage sample size, see the discussion and example in \cite[p.~115]{MukSilvaBook2009}. 
 
 But, indeed, Mukhopadhyay's argument also applies when using $ S_{n_0}^2 + f^2/n_0 $, for some $ f > 0 $, and then one gets the lower bound $ \max\{ \overline{N}_0, t(\overline{N}_0-1)_{1-\alpha/2} f / d \} $ and $ \max\{ \overline{N}_0, \Phi^{-1}(1-\alpha/2) f / d \} $, respectively. It is easy to check that all above arguments go through as well, if we replace $ S_n^2 $ by any guess or pilot estimate using a (very) small sample. 
 
\textbf{Three-Observation-Rule:} Since frequently in applications it is possible to sample at least three observations, we propose to choose $ f $ as an estimate $ \wh{\sigma}_{\mu,3} $ of $ \sigma_\mu$, using three addtional observations, on which we condition in what follows. This leads to our proposal for the first-stage sample size, namely
\begin{equation}
	\label{DefN0}
	N_0 = \max \left\{ \overline{N}_0, \left \lfloor \frac{\Phi^{-1}(1-\alpha/2) \wh{\sigma}_{\mu,3}  }{d} \right \rfloor + 1 \right\}.
\end{equation}
Note that $ N_0 $ depends on the preassigned precision $d$ and satisfies $ N_0 \to \infty $, as $ d \downarrow 0 $. 

It is natural to estimate $ N_{opt}^* $ by $ \wh{N}_{opt}^* = \frac{\wh{\sigma}_{N_0}^2 \Phi^{-1}(1-\alpha/2)^2 }{ d^2 } $, and this leads to the final sample size of the procedure,
\begin{equation}
	\label{DefN1}
	\wh{N}_{opt} = \max \left\{  N_0, \biggl\lfloor \frac{ \wh{\sigma}_{N_0}^2 \Phi^{-1}(1-\alpha/2)^2 }{d^2} \biggr\rfloor + 1 \right\},
\end{equation}
which is a random variable (depending on the first-stage data), as $ \wh{\sigma}_{N_0}^2 $ estimates the asymptotic variance of the estimator $ \wh{\mu}_{N_0} $ using the first-stage sample of size $N_0$. Observe that we continue to add a $ * $ in notation to indicate quantities which may not be integer-valued.

Let us briefly review the following facts and considerations leading to the notions  of consistency and unbiasedness: Note that $ N_{opt}^* \to \infty $ and $ \wh{N}_{opt}^* \to \infty $, in probability, for any (arbitrary) weakly consistent estimator $ \wh{\sigma}_\mu^2 $ of the asymptotic variance $ \sigma_{N_0}^2 $ (based on the first-stage sample of size $ N_0 $), which slighly complicates their comparison. If we only know that $ | \wh{\sigma}_\mu^2 - \sigma_\mu^2 | = o_P(1) $, which follows from ratio consistency $ | \wh{\sigma}_\mu^2 / \sigma_\mu^2 - 1 |  = o_P(1) $, then the difference $  \wh{N}_{opt}^* - N_{opt}^*  $ is not guaranteed to be bounded, as 
\[
| \wh{N}_{opt}^* - N_{opt}^* | = | \wh{\sigma}_\mu^2 - \sigma_\mu^2 | \frac{ \Phi^{-1}( 1 - \alpha/2 )^2 }{ d^2 },
\]
where the first factor is $ o_P(1) $, but the second one diverges, as $ d \downarrow 0 $. For this reason, $ \wh{N}_{opt}^* $  is called {\em consistent} for the asymptotically optimal sample size $ N_{opt}^* $, if the ratio approaches $1$ in probability, i.e. if 
\[
\frac{ \wh{N}_{opt}^* }{ N_{opt}^*  } = 1 + o_P(1),
\]
as $ d \downarrow 0 $. Having in mind that the second-stage (final) sample size $ \wh{N}_{opt}^* $ is random whereas the unknown optimal value, $ N_{opt}^* $, is non-random, the question arises whether $\wh{N}_{opt}^* $ is, at least, close ot $ N_{opt}^* $ on average. Therefore, to address this question and going beyond consistency, $ \wh{N}_{opt}^* $ is called {\em asymptotically first order efficient in the sense of Chow and Roberts}, if
\[
\frac{ E( \wh{N}_{opt}^* ) }{ N_{opt}^* } = 1 + o(1),
\] 
as $ d \downarrow 0 $. 

Observe that the last property allows for the case that the difference $ \wh{N}_{opt}^* - N_{opt}^* $ tends to $ \infty $ even in the mean, as $ d $ tends to $0$. This typically indicates that the procedure is substantially oversampling the optimal sample size. A procedure for which the estimated optimal same size remains on average in a bounded vicinity of the optimal truth is, of course, preferable.  $ \wh{N}_{opt}^* $ is called {\em second order asymptotically efficient}, if
\[
  E(  \wh{N}_{opt}^* - N_{opt}^* ) = O(1),
\]
as $ d \downarrow 0 $. 

The regularity assumptions we need to impose are as follows:

\textbf{Assumption (E):} The estimator $ \wh{\sigma}_{N_0}^2 $ is consistent and asymptotically unbiased for $ \sigma_{\mu}^2 $, i.e. 
\[
\frac{\wh{\sigma}_{N_0}^2 }{ \sigma_{\mu}^2 } = 1 + o_P(1), \qquad \frac{ E( \wh{\sigma}_{N_0}^2  )}{ \sigma_\mu^2}  = o(1),
\]
as $ d \downarrow 0 $.

This assumption is not restrictive and satisfied by many estimators. For example, the jackknife variance estimator studied in \cite{Shao1993}, \cite{ShaoWu1998} and \cite{StelandChang2019} provides an example satisfying Assumption (E). 

Further, we require the following strengthening of the central limit theorem to hold.

\textbf{Assumption (A):}  $ \wh{\mu}_{n} $ satisfies the random central limit theorem, i.e. for any family $ \{ N_a : a > 0 \} $ of stopping times for which $ N_a / a \stackrel{P}{\to} k $, $ 0 < k < \infty $, it holds
\[
  \sqrt{N_a} ( \wh{\mu}_{N_a} - \mu )/ \sqrt{\sigma_\mu^2 a k } \stackrel{d}{\to} N(0, 1),
\]
as $ a \to \infty $.

The validity of the random central limit theorem is required, as we have to employ a normal approximation with the first-stage sample size, which is random by construction. Clearly, however, for i.i.d. observations following an arbitrary distribution with finite second moment and $ \wh{\mu}_n $ the arithmetic mean Assumption (A) is well known, see, e.g.,  \cite[Th.~2.7.2]{GoshMukhoSen1997}. 

The following theorem summarizes the main asymptotic first order properties of the proposed two-stage approach to construct a fixed width confidence interval. 

\begin{theorem} Suppose that  Assumption (E) is satisfied. Then the following two assertions hold true.
	\label{Th_FW_CI}	
	\begin{itemize}
		\item[(i)]
		The estimated optimal sample size $ \wh{N}_{opt} $ is consistent for $ N_{opt}^* $, i.e.
		\[
		\frac{\wh{N}_{opt}}{N_{opt}^*} = 1 + o_P(1),
		\]
		as $ d \downarrow 0 $.
		\item[(ii)] 
		$ \wh{N}_{opt}  $ is asymptotically first order efficient for $ N_{opt}^* $, i.e.
		\[
		\frac{ E( \wh{N}_{opt} ) }{ N_{opt}^*  } = 1 + o(1),
		\]
		as $ d \downarrow 0 $.
	\end{itemize}
	If, in addition, Assumption (A) holds, then we have:
	\begin{itemize}
		\item[(iii)] The fixed-width confidence interval $ I_{\wh{N}_{opt}} $ has asymptotic
		coverage $ 1 - \alpha $, i.e.
		\[
		P\left( I_{{\wh{N}_{opt}}} \ni \mu \right) 
		= 1 - \alpha + o(1),
		\]
		as $ d \downarrow 0 $.
	\end{itemize}
\end{theorem}

\begin{remark}
	It is worth mentioning that the proof of Theorem~\ref{Th_FW_CI} (i)-(ii) shows the following stronger properties
	\begin{itemize}
		\item[(i)] $ \wh{N}_{opt} $ is consistent for $ N_{opt}^* $, if and only if $ \wh{\sigma}_{N_0}^2 $ is consistent for $ \sigma_{\mu}^2 $.
		\item[(ii)] $ \wh{N}_{opt} $ is asymptotically unbiased for $ N_{opt}^* $, if and only if $ \wh{\sigma}_{N_0}^2 $ is asymptotically unbiased for $ \sigma_{\mu}^2 $.
	\end{itemize}	
\end{remark}

Let us now discuss the second order properties of the fully nonparametric procedure. In the literature, so far second order efficiency for the problem at hand has been studied for parametric (Gaussian) models, see \cite{MukDuggan1999}, leading to a known distribution of $ \wh{N}_{opt} $, a chi-squared distribution induced by the fact that the sample variance follows a chi-squared law, which converges to the normal law if $ d \to 0 $. To achieve second order efficiency, the probability  probability $ P( \wh{N}_{opt} = N_{opt}^* ) $ that the sample size is not increased at the second stage needs to decrease faster than the first-stage sample size $ N_0 $. In a parametric setting that probability can be handled and estimated by means of appropriate Taylor expansions using properties of the known distribution function. 

In a fully nonparametric framework the exact distribution is unknown to us and estimating the probability under the limiting law is not sufficient, since we have to take into account the error of approximation. But due to the Berry-Esseen bound the error is of the order $ O( N_0^{-1/2} ) $.  Therefore, the following result proceeds in a different way than the proofs for parametric settings and bounds the probability $ P( \wh{N}_{opt} = N_{opt}^* ) $ using nonparametric techniques.

\begin{theorem} 
\label{2ndEff1}
	Assume that Assumption (E) holds and $ Y_1, Y_2, \dots $ are i.i.d. with $ E (Y_1^8 ) < \infty $. Then the two-stage procedure given by $ \wh{N}_{opt}^* $ is second order efficient, i.e.
	\[
	  E( \wh{N}_{opt}^* ) - N_{opt}^* = O(d),
	\]
	as $ d \downarrow 0 $.
\end{theorem}

\subsection{Fixed-width interval under high-confidence asymptotics}

Theorem~\ref{Th_FW_CI} establishes the validity of the proposed sampling strategy {\em for small accuracy $d$}, i.e. in a {\em high-accuracy framework}: The (asymptotic) first order properties hold, if $d$ tends to zero. To some extent, this is counter-intuitive as we aim at constructing a {\em fixed-width} confidence interval and, for applications, we are then essentially limited to confidence statement when $d$ is small. 

In some applications, however, $d$ may be not small (enough) but one aims at ensuring the confidence  statement that the interval covers the true parameter with {\em high confidence}. This suggests to consider the case that $d$ is fixed and $ 1-\alpha $ tends to $1$ (or equivalently $ \alpha \to 0 $). That type of asymptotics may be of particular importance in fields such as statistical genetics or brain research, where it is common to use very small significance levels $ \alpha $.

Recalling formula (\ref{Formula_N_opt_star}) for the asymptotically optimal (unknown) sample size $ N_{opt}^* $ and noticing that, for fixed $d$, $ N_{opt}^* \to \infty $ holds if $ 1-\alpha/2 \to 1 $, again justifying the application of the CLT, the question arises whether consistency and efficiency can be established under this different asymptotic regime. 

To begin, let us notice that the notions of {\em consistency}, {\em asymptotically first and second order efficiency} and {\em asymptotic coverage} can be defined analogously by simply replacing the limits $ \lim_{d \to 0} $ by $ \lim_{1-\alpha\to1} $ for fixed accuracy $d$.

The following theorem asserts that the proposed methodology is valid without any modification of $ \widehat{N}_{opt} $ under the high-confident asymptotics, although the proof differs. 

\begin{theorem} 
\label{Th_2st_alpha}
	\begin{itemize}
		\item[(i)] Assume that (E) holds. Then $ \wh{N}_{opt} $ is consistent for $ N_{opt}^* $, i.e.
		\[
			\frac{\wh{N}_{opt}}{N_{opt}^*} = 1 + o_P(1),
		\]
		 as $ 1-\alpha \to 1 $, and asymptotically first order efficient, i.e.
		\[
			\frac{ E( \wh{N}_{opt} ) }{ N_{opt}^*  } = 1 + o(1),
		\]
		as $ 1-\alpha \to 1 $.
		\item[(ii)] If Assumptions (E) and (A) hold, then $I_{\wh{N}_{opt}}  $ has asymptotic coverage $ 1 - \alpha $, i.e.
		\[
		  \lim_{1-\alpha\to 1} \left| P\left( I_{\wh{N}_{opt}} \ni \mu  \right)  - (1- \alpha) \right| = 0.
		\]
	\end{itemize}
\end{theorem}

\begin{remark}
\label{ExtensionN0Gamma} The assertions of Theorem~\ref{Th_2st_alpha} also hold true, if
	the first-stage sample size $ N_0 $ is defined as 
	\[
	  N_0 = \max\left\{ \overline{N}_0, \left\lfloor \left( \frac{ f \Phi^{-1}(1-\alpha/2)}{d} \right)^{2/(1+\gamma)} \right\rfloor  \right\}
	\]
	for some $ 0 < \gamma < \infty $ and an arbitrary given constant $ f > 0 $, as proposed in \cite{Muk1980} (with $ f = 1 $).
\end{remark}

The question arises whether the procedure exhibits second order efficiency under the high-accuracy regime as well. The answer is positive. 

\begin{theorem} 
\label{2ndEff}
	Assume that $ Y_1, Y_2, \dots $ are i.i.d. with $ E (Y_1^8 ) < \infty $. Then the two-stage procedure given by $ \wh{N}_{opt}^* $ is second order efficient. Precisely, we have
	\[
	| E( \wh{N}_{opt}^* ) - N_{opt}^* | = O(1), 
	\]
	as $ 1-\alpha \to 1 $. 
\end{theorem}

\subsection{Proofs}

\begin{proof}[Proof of Theorem~\ref{Th_FW_CI}]  First, notice that, by definition of $ \wh{N}_{opt} $,
	\begin{equation}
		\label{LowerBound}
		\wh{N}_{opt} \ge \left\lfloor \frac{ \wh{\sigma}_{N_0}^2 \Phi^{-1}(1-\alpha/2)^2 }{ d^2 } \right\rfloor +1 \ge \frac{ \wh{\sigma}_{N_0}^2 \Phi^{-1}(1-\alpha/2)^2 }{ d^2 }.
	\end{equation}
	It is easy to see that $ \max \{ \trunc{z}, a \} \le z + a $ for all nonnegative real $ z $ and any positive constant $a$. 	Therefore, we have
	\begin{align*}
		\wh{N}_{opt} 
		& = \max \left\{ N_0, \biggl\lfloor \frac{ \wh{\sigma}_{N_0}^2 \Phi^{-1}(1-\alpha/2)^2 }{ d^2 }  + 1 \biggr\rfloor  \right\}  \\
		& \le N_0 + \frac{ \wh{\sigma}_{N_0}^2 \Phi^{-1}(1-\alpha/2)^2 }{ d^2 } + 1.
	\end{align*}
	Plugging in the definition of $ N_0 $ we further obain
	\begin{align*}
		\wh{N}_{opt} & \le \max \left\{
		\biggl\lfloor \frac{ \wh{\sigma}_{\mu,3} \Phi^{-1}(1-\alpha/2) }{d} \biggr\rfloor, \overline{N}_0 \right\} 
		+ \frac{ \wh{\sigma}_{N_0}^2 \Phi^{-1}(1-\alpha/2)^2 }{ d^2 } + 1 \\
		& \le \frac{ \wh{\sigma}_{N_0}^2 \Phi^{-1}(1-\alpha/2)^2 }{ d^2 }   + \frac{ \wh{\sigma}_{\mu,3} \Phi^{-1}(1-\alpha/2) }{d} + \overline{N}_0 + 1.
	\end{align*}
	Combining the last estimate with (\ref{LowerBound}), we arrive at 
	\[
	\frac{ \wh{\sigma}_{N_0}^2 \Phi^{-1}(1-\alpha/2)^2 }{ d^2 } \le 
	\wh{N}_{opt}
	\le 
	\frac{ \wh{\sigma}_{N_0}^2 \Phi^{-1}(1-\alpha/2)^2 }{ d^2 }   
	+ \frac{ \wh{\sigma}_{\mu,3} \Phi^{-1}(1-\alpha/2) }{d} + \overline{N}_0 + 1, 
	\]
	which implies, due to (\ref{Formula_N_opt_star}), 
	\begin{equation}
		\label{InequalityRatio}
		\frac{ \wh{\sigma}_{N_0}^2 }{ \sigma_\mu^2 } \le 
		\frac{ \wh{N}_{opt} }{ N_{opt}^* } \le
		\frac{ \wh{\sigma}_{N_0}^2 }{ \sigma_\mu^2 } 
		+ \frac{d \wh{\sigma}_{\mu,3} }{ \Phi^{-1}(1-\alpha/2)  \sigma_\mu^2 } 
		+ \frac{(\overline{N}_0+1)d^2}{\sigma_\mu^2 \Phi^{-1}(1-\alpha/2)^2}.
	\end{equation}
	We are led to
	\begin{equation}
		\label{InequalityRatio2}
		\left| \frac{ \wh{N}_{opt} }{ N_{opt}^* }  - 1 \right|
		\le \left| \frac{ \wh{\sigma}_{N_0}^2 }{ \sigma_\mu^2 } - 1 \right| 
		+ \frac{d \wh{\sigma}_{\mu,3} }{ \Phi^{-1}(1-\alpha/2)  \sigma_\mu^2 } 
		+ \frac{ (\overline{N}_0+1)d^2 }{ \sigma_\mu^2 \Phi^{-1}(1-\alpha/2)^2 }.
	\end{equation}
	Recalling that $ A_n \le X_n \le B_n $ with $ A_n, B_n = o_P(1) $ implies $ X_n = o_P(1) $, since for any $ \delta > 0 $ we have $ P( |X_n| > \delta ) \le P( A_n \le - \delta ) + P( B_n > \delta ) = o(1) $, as $ n \to \infty $, the first assertion follows: 
	By (\ref{InequalityRatio2}) $ \wh{N}_{opt} $ is consistent for $ N_{opt}^* $ if  $ \wh{\sigma}_{N_0}^2  $ is consistent for $ \sigma_\mu^2 $, as $ d \downarrow 0 $. (\ref{InequalityRatio}) also immediately yields
	\[
	  \left| \frac{ \wh{\sigma}_{N_0}^2 }{ \sigma_\mu^2 }  - 1 \right|  \le \left| \frac{ \wh{N}_{opt} }{ N_{opt}^* } - 1 \right| 
	\]
	which shows that 'only if' holds as well.
	Next, taking expectations in (\ref{InequalityRatio}) implies the result on the asymptotically unbiasedness. It remains to show that the fixed width confidence interal has asymptotic coverage $ 1 - \alpha $.     
	First observe that
	\begin{align*}
	P\left( I_{{\wh{N}_{opt}}} \ni \mu \right) &= P \left( - d \frac{\sqrt{ \wh{N}_{opt}}}{\sigma_\mu}
	\le \sqrt{ \wh{N}_{opt} } \frac{ \wh{\mu}_{\wh{N}_{opt}} - \mu }{ \sigma_\mu }
	\le d \frac{\sqrt{ \wh{N}_{opt}}}{\sigma_\mu}
	\right) \\
	  & =  P \left( - d \frac{\sqrt{ N_{opt}^*}}{\sigma_\mu}
	  \le \sqrt{ \frac{ N_{opt}^*}{ \wh{N}_{opt} } } \sqrt{ \wh{N}_{opt} } \frac{ \wh{\mu}_{\wh{N}_{opt}} - \mu }{ \sigma_\mu }
	  \le d \frac{\sqrt{ N_{opt}^*}}{\sigma_\mu}
	  \right).
	\end{align*}
	Define
	\[
	  H_d(x) = P\left( \sqrt{ \frac{ N_{opt}^*}{ \wh{N}_{opt} } } \sqrt{ \wh{N}_{opt} } \frac{ \wh{\mu}_{\wh{N}_{opt}} - \mu }{ \sigma_\mu } \le x   \right), \qquad x \in \R.
	\]
	By Assumption (A) and Slutzky's lemma, we have
	\[
	  \sup_{x \in \R} \left| H_d(x) - \Phi(x) \right| \to 0,
	\]
	as $ d \downarrow 0 $. Now, by definition of $ H_d $ and $ N_{opt}^* $ and linearity of the function evaluation $ f \mapsto  f \bigr|_{a}^{b} = f(b) - f(a) $ (for fixed reals $a,b \in D $) for a function defined on $ D \subset \mathbb{R} $,
	\begin{align*}
	  P\left( I_{{\wh{N}_{opt}}} \ni \mu \right)  
	  &= H_d \biggr|_{-\Phi^{-1}(1-\alpha/2)}^{\Phi^{-1}(1-\alpha/2)} \\
	  &= (H_d - \Phi + \Phi )\biggr|_{-\Phi^{-1}(1-\alpha/2)}^{\Phi^{-1}(1-\alpha/2)} \\	  
	  &= (H_d - \Phi) \biggr|_{-\Phi^{-1}(1-\alpha/2)}^{\Phi^{-1}(1-\alpha/2)} + (1-\alpha).
	\end{align*}
	Therefore we obtain
	\[
	  \left|  P\left( I_{{\wh{N}_{opt}}} \ni \mu \right) - (1-\alpha)  \right| \le 2 \sup_{x \in \R} \left| H_d(x) - \Phi(x) \right| \to 0,
	\]
	as $ d \downarrow 0 $, which completes the proof.
\end{proof}

\begin{proof}[Proof of Theorem~\ref{2ndEff1}]
	First observe that $ N_0 \to \infty $ if and only if $ d \to 0 $, so that we can show the result for $ N_0 \to \infty $.
	We use the refined basic inequality
	\[
	\frac{ \wh{\sigma}_{N_0}^2 \Phi^{-1}(1- \alpha/2)^2 }{ d^2 } \le \wh{N}_{opt}^* \le  N_0 \eins( \wh{N}_{opt}^*  = N_0 ) +
	\frac{ \wh{\sigma}_{N_0}^2 \Phi^{-1}(1- \alpha/2)^2 }{ d^2 } + 1,
	\]
	which implies in view of Assumption (E)
	\[
	N_{opt}^* + o(1) \le E( \wh{N}_{opt}^* ) \le N_0 P(  \wh{N}_{opt}^*  = N_0 ) + N_{opt}^* + 1 + o(1).
	\]
	This means,
	\[
	o(1) \le E( \wh{N}_{opt}^* )  - N_{opt}^* \le N_0 P(  \wh{N}_{opt}^*  = N_0 ) + 1 + o(1),
	\]
	and the result follows if we show that
	\[
	P(  \wh{N}_{opt}^*  = N_0 )  = O( N_0^{-1} ).
	\]
	We may assume that $ N_0 $ is large enough to ensure that $ N_0 = \lfloor  \frac{\wh{\sigma}_{\mu,3} \Phi^{-1}(1-\alpha/2 ) }{d} \rfloor + 1 $. Since by definition of $ \wh{N}_{opt}^* $ 
	\[
	\wh{N}_{opt}^*  = N_0 \Leftrightarrow  \left\lfloor  \frac{\wh{\sigma}_{\mu,3} \Phi^{-1}(1-\alpha/2 ) }{d} + 1 \right\rfloor + 1 
	\ge \left \lfloor  \frac{\wh{\sigma}_{N_0} \Phi^{-1}(1-\alpha/2 ) }{d^2} \right \rfloor + 1,  
	\]
	and using the elementary estimates  $ \lfloor x \rfloor \ge x - 1 $ and $ \lfloor x \rfloor + 1 \le x + 1 $, if $ x \ge  0 $,  we obtain
	\begin{align*}
	P(  \wh{N}_{opt}^* = N_{opt}^* ) & 
	\le P\left( \frac{ \wh{\sigma}_{N_0}^2 \Phi^{-1}(1- \alpha/2)^2 }{ d^2 }   \le \frac{\wh{\sigma}_{\mu,3} \Phi^{-1}(1-\alpha/2 ) }{d} +1 \right) \\
	& = P\left( \wh{\sigma}_{N_0}^2 - \sigma^2 \le - \sigma^2 +
	\left(  \frac{\wh{\sigma}_{\mu,3} }{\Phi^{-1}(1-\alpha/2)} + \frac{d}{\Phi^{-1}(1-\alpha/2)^2} \right)  d  \right).
	\end{align*}
	Since $ \left(  \frac{\wh{\sigma}_{\mu,3} }{\Phi^{-1}(1-\alpha/2)} + \frac{d}{\Phi^{-1}(1-\alpha/2)^2} \right)  d  = o(1) $,  as $ d \downarrow 0 $, we obtain
	\[
	P(  \wh{N}_{opt}^* = N_{opt}^* )  = P( \sigma^2 - \wh{\sigma}_{N_0}^2 \ge \sigma^2 + o(1) ) \le P\left(  \sigma^2 - \wh{\sigma}_{N_0}^2  \ge \frac{\sigma^2}{2}  \right).
	\]
	But if $ E( Y_1^8 ) < \infty $, then
	\[
	P\left(  \sigma^2 - \wh{\sigma}_{N_0}^2 \ge \frac{\sigma^2}{2}  \right) \le
	\frac{E( \wh{\sigma}_{N_0}^2 - \sigma)^4}{(\sigma^2/2)^4} = O( N_0^{-2} ),
	\]
	see e.g. [Lemmas~A.1 and A.2]\cite{ChoSteland2016}. Hence the assertion follows.
\end{proof}

Let us now prove the corresponding results under the high-confident asymptotic framework.

\begin{proof}[Prood of Theorem~\ref{Th_2st_alpha}]
	Repeating the purely algebraic calculations from above we again obtain (\ref{InequalityRatio2}) for any $d$:
	\[
		\left| \frac{ \wh{N}_{opt} }{ N_{opt}^* }  - 1 \right|
		\le \left| \frac{ \wh{\sigma}_{N_0}^2 }{ \sigma_\mu^2 } - 1 \right| + \frac{d \wh{\sigma}_{\mu,3} }{ \Phi^{-1}(1-\alpha/2)} + 
		\frac{(\overline{N}_0+1)d^2}{\sigma_\mu^2 \Phi^{-1}(1-\alpha/2)}.
	\]
	Clearly, the second and third term are $ o(1) $, if $ d $ is fixed and $ 1-\alpha \to 1 \Leftrightarrow \alpha \downarrow 0 $. Further, noting that $ N_0 \to \infty $, if $ d $ is fixed and $ \alpha \downarrow 0 $, the first term is $ o_P(1) $, if $ d $ is fixed and $ \alpha \downarrow 0 $.
	Taking expectations similar arguments apply. Hence (i) is shown. To establish (ii), first observe that if $ N_a $, $ a > 0 $, 
	is an arbitrary family of stopping times with $ N_a / a \stackrel{P}{\to}  k \in (0, \infty) $, then $ N_a / ak \stackrel{P}{\to} 1 $, as  $ a \to \infty $,  and therefore by Slutzky's lemma and the RCLT,
	\[
	  \sqrt{ak} \frac{ \wh{\mu}_{N_a} - \mu }{ \sigma_\mu } \stackrel{d}{\to} N( 0, 1 ), 
	\]
	as $ a \to \infty $. Now consider $ \wh{N}_{opt} $ as a family of stopping times parameterized by $ a = \Phi^{-1}(1-\alpha/2)^2 $. Obviously, $ a \to \infty \Leftrightarrow \alpha \downarrow 0 $. Then, using (i) and recalling (\ref{Formula_N_opt_star}), $ \wh{N}_{opt} / a \stackrel{P}{\to} k = \sigma_\mu^2 / d^2 $, as $ a \to \infty $. Since $ d $ is fixed, we have $ 0 < k < \infty $. Consequently,
	\begin{equation}
	\label{AsNT_N_star}
	  T_{\wh{N}_{opt}} = \sqrt{ \frac{ \Phi^{-1}(1-\alpha/2)^2 \sigma_\mu^2}{ d^2 }}  \frac{ \wh{\mu}_{\wh{N}_{opt}} - \mu  }{ \sigma_\mu } 
	  \stackrel{d}{\to} N(0,1),
	\end{equation}
	as $ a \to \infty $. Since $ T_{\wh{N}_{opt}} = \frac{ \Phi^{-1}(1-\alpha/2) }{ d } (\wh{\mu}_{\wh{N}_{opt}} - \mu) $, it follows that
	\begin{align*}
	  P\left( I_{{\wh{N}_{opt}}} \ni \mu \right) & = P( -d \le \wh{\mu}_{\wh{N}_{opt}} - \mu \le  d) \\
	    & = P\left( -\Phi^{-1}(1-\alpha/2) \le  \sqrt{ \frac{ \Phi^{-1}(1-\alpha/2)^2 \sigma_\mu^2}{ d^2 }}  \frac{ \wh{\mu}_{\wh{N}} - \mu  }{ \sigma_\mu }  \le \Phi^{-1}(1-\alpha/2) \right) \\
	    & = H_a \biggr|_{-\Phi^{-1}(1-\alpha/2)}^{\Phi^{-1}(1-\alpha/2)}, 
	\end{align*}
	where $ H_a $ denotes the distribution function of $ T_{\wh{N}_{opt}} $. By (\ref{AsNT_N_star}) and Polya's theorem we obtain
	\[
	  \lim_{a\to \infty} \sup_{x \in \R} | H_a(x) - \Phi(x) | = 0.
	\]
	Now we can conclude that
	\[
	  \left| P\left( I_{{\wh{N}_{opt}}} \ni \mu \right)  - (1-\alpha ) \right| \le \left| (H_a-\Phi)_{-\Phi^{-1}(1-\alpha/2)}^{\Phi^{-1}(1-\alpha/2)}  \right| 
	  \le 2 \sup_x | H_a(x) - \Phi(x) |  \to 0,
	\]
	as $ a \to \infty $, which completes the proof.
\end{proof}

\begin{proof}[Proof of Remark~\ref{ExtensionN0Gamma}]
	The proof of the consistency needs a minor modification. Plugging in the new definition of $ N_0 $ we obtain
	\begin{align*}
	\wh{N}_{opt} & \le \max \left\{
	\biggl\lfloor \left( \frac{ f \Phi^{-1}(1-\alpha/2) }{d} \right)^{2/(1+\gamma)} \biggr\rfloor, \overline{N}_0 \right\} 
	+ \frac{ \wh{\sigma}_{N_0}^2 \Phi^{-1}(1-\alpha/2)^2 }{ d^2 } + 1 \\
	& \le \frac{ \wh{\sigma}_{N_0}^2 \Phi^{-1}(1-\alpha/2)^2 }{ d^2 }   +
	\left(  \frac{ f \Phi^{-1}(1-\alpha/2) }{d} \right)^{2/(1+\gamma)} + \overline{N}_0 + 1.
	\end{align*}
	Combining the last estimate with (\ref{LowerBound}), we arrive at 
	\[
	\frac{ \wh{\sigma}_{N_0}^2 \Phi^{-1}(1-\alpha/2)^2 }{ d^2 } \le 
	\wh{N}_{opt}
	\le 
	\frac{ \wh{\sigma}_{N_0}^2 \Phi^{-1}(1-\alpha/2)^2 }{ d^2 }   
	+ \left( \frac{ f \Phi^{-1}(1-\alpha/2) }{d} \right)^{2/(1+\gamma)} + \overline{N}_0 + 1, 
	\]
	which implies, due to (\ref{Formula_N_opt_star}), 
	\begin{equation}
	\label{InequalityRatio}
	\frac{ \wh{\sigma}_{N_0}^2 }{ \sigma_\mu^2 } \le 
	\frac{ \wh{N}_{opt} }{ N_{opt}^* } \le
	\frac{ \wh{\sigma}_{N_0}^2 }{ \sigma_\mu^2 } 
	+ \frac{f^{2/(1+\gamma)}}{\sigma_\mu^2} d^{2\gamma/(1+\gamma)} \Phi^{-1}(1-\alpha/2)^{-2\gamma/(1+\gamma)}
	+ \frac{(\overline{N}_0+1)d^2}{\sigma_\mu^2 \Phi^{-1}(1-\alpha/2)^2}.
	\end{equation}
	We are led to
	\begin{equation}
	\label{InequalityRatio2}
	\left| \frac{ \wh{N}_{opt} }{ N_{opt}^* }  - 1 \right|
	\le \left| \frac{ \wh{\sigma}_{N_0}^2 }{ \sigma_\mu^2 } - 1 \right| 
	+ \frac{ f^{2/(1+\gamma)}}{\sigma_\mu^2} d^{2\gamma/(1+\gamma)} \Phi^{-1}(1-\alpha/2)^{-2\gamma/(1+\gamma)}
	+ \frac{ (\overline{N}_0+1)d^2 }{ \sigma_\mu^2 \Phi^{-1}(1-\alpha/2)^2 }.
	\end{equation}		
	Since $ \gamma > 0$ ,the second term is $ o_P(1) $, if $ 1-\alpha \to 1 $.
\end{proof}

\begin{proof}[Proof of Theorem~\ref{2ndEff}] Observing that $ N_0 \to \infty $ if and only if $ 1-\alpha \to 1 $ and $ O(N_0^{-1}) = O( \Phi^{-1}(1-\alpha/2)) $, the proof of Theorem~\ref{2ndEff1} carries over and provides the bound
	\[
	| E( \wh{N}_{opt}^*  ) -  N_{opt}^* | \le N_0 O( N_0^{-2}) + 1 = O(N_0^{-1}) + 1.
	\]
\end{proof}

\section{Application to a projection of high-dimensional data}
\label{Sec: HDData}

An interesting application is to study the construction of a fixed-width confidence interval for the mean projection of high-dimensional data. As an example, we may ask how for how long we need to observe the asset returns of stocks associated to a portfolio given by portfolio vector $ \vecw_n $,  in order to set up a $(1-\alpha)$-confidence interval for the mean portfolio return $ \vecw_n' \bfmu $ with a precision $d$. Such uncertainty quantification of a mean projection also arises when projecting a data vector to reduce dimensionality, calculating a projection $ \vecw_n'\vecX_n $ of a $p_n $-dimensional data vector $ \vecX_n $ being a common approach to handle multivariate data when the dimension of the observed vectors is large. Widely used methods are PCA, where one projects onto eigenvectors of the (estimated) covariance matrix of the data, sparse PCA yielding sparse projection vectors or LASSO regressions. For the latter approach, recall that a LASSO regression determines a sparse weighting vector, the regression coefficients, such that the associated projection of the regressors provides a good explanation of the response variable. 

Of particular interest is the situation that the dimension gets larger when the sample size increases, in order to mimic the case of large dimension when relying on asymptotics where the sample size increases. Then the weighting vectors depend on the sample size. We show that the general assumptions established in the previous section apply under mild uniform integrability conditions on $ \vecw_n'\vecX_n $, $ n \ge 1 $.

\subsection{Procedure}

Suppose we are given a series of random vectors of increasing dimension, such that at time instant $n$ we have at our disposal $n$ i.i.d. random vectors 
\[
	\vecX_{n1}, \dots, \vecX_{nn} \sim (\bfmu_n, \bfSigma_n ),
\]
of dimension $ p = p_n $, where $ p_n \to \infty $, as $ n \to \infty $ is allowed. Here the notation $ \bfxi_n  \sim (\bfmu_n, \bfSigma_n) $ means that the random vector $ \bfxi_n $ follows a distribution with mean vector $ \bfmu_n  $ and covariance matrix $ \bfSigma_n $. We are interested in an asymptotic fixed-width confidence interval for the projected mean vector,
\[
	\theta_n = \vecw_n' \bfmu_n,
\]
i.e. for a preassigned accuracy $ d > 0 $ and a given confidence level $ 1-\alpha $ we aim at finding a sample size $N$ such that the confidence interval $ [T_N - d, T_N + d ] $ as asymptotic confidence $ 1-\alpha $, 
\[
P( T_N - d \le  \vecw_N' \bfmu_N \le T_N +d  ) = 1-\alpha + o(1), 
\] 
as $ d \downarrow 0$. Here, for any (generic) sample size $n$, $ T_n $ is an estimator of $ \vecw_n' \bfmu $. For simplicity, we consider the unbiased estimator $T_n = \vecw_n' \overline{\vecX}_n $ where $ \overline{\vecX}_n = \frac{1}{n} \sum_{i=1}^n \vecX_{ni} $. Denote by $ T_n^* = (T_n - E(T_n)) / \sqrt{ \Var(T_n) }  $ the standardized version.

The proposed two-stage procedure is as in the previous section: At the first stage draw
\begin{equation}
N_0 = \max \left\{ \overline{N}_0, \left \lfloor \frac{\Phi^{-1}(1-\alpha/2) \wh{\sigma}_{in}  }{d} \right \rfloor + 1 \right\}.
\end{equation}
observations, where $ \overline{N}_0  $ is a given minimal sample size and $\wh{\sigma}_{in} $ an initial estimate of the standard deviation of the projections using indepdent pilot data, such as the three-observations-rule estimator. Next, we estimate the variance of the projections from the first-stage sample by $\wh{\sigma}_{N_0}^2 $ and calculate the final sample size
\begin{equation} 
\label{DefN1}
\wh{N}_{opt} = \max \left\{  N_0, \biggl\lfloor \frac{ \wh{\sigma}_{N_0}^2 \Phi^{-1}(1-\alpha/2)^2 }{d^2}  \biggr\rfloor + 1 \right\}.
\end{equation}

In practice, one considers a certain number of variables and so we assume that these variables are observable for the relevant sample sizes $ N_0 $ and $ \wh{N}_{opt} $ as well as that the projection vectors only have non-zero entries for those variables of interest. The mathematical framework allowing for an increasing dimension is used to justifiy the procedure when the number of variables resp. dimension is large compared to the sample sizes in use.

Let us assume that
\begin{equation}
\label{ConvergenceVarianceProj}
\lim_{n \to \infty} \vecw_n' \bfSigma_n \vecw_n = \sigma^2_{\vecw} > 0
\end{equation}
and for $ r \in \{ 1, 2 \} $ 
\begin{equation}
\label{UniformInt}
\sup_{n \ge 1 }  E( | \vecw_n'\vecX_{n1} |^r \eins( | \vecw_n'\vecX_{n1} |^k > c ) ) \to 0, \quad c \to \infty, \quad k = 1, \dots, r. 
\end{equation}

Assumption (\ref{ConvergenceVarianceProj}) is mild and rules out cases where the variance of the projection vanishes asymptotically. The uniform integrability required in (\ref{UniformInt}) is a crucial technical condition to ensure that the weak law of large numbers and the central limit theorem apply. In many cases the condition can be formulated in terms of the $ \vecX_n $: For that purpose, suppose additionally that the norms $ W_n = \| \vecw_n \|_2 $  are bounded by constant $ W < \infty $. Then, using the simple fact that $ | \vecw_n' \vecX_n | \le W \| \vecX_n \|_2 $, condition (\ref{UniformInt}) holds, if the $ \vecX_n $ satisfy 
\begin{equation}
\label{UniformInt_X}
  \sup_{n \ge 1 } E( \| \vecX_{n1} \|_2^r \eins( \| \vecX_{n1} \|_2^k > c ) ) \to 0, \quad c \to \infty, \quad k = 1, \dots, r. 
\end{equation}
A simple sufficient condition for (\ref{UniformInt_X}) is the moment condition
\begin{equation}
\label{Moment_X}
  \sup_{n \ge 1} E \| \vecX_n \|_2^{r+\delta} < \infty
\end{equation}
for some $ \delta > 0 $. 

We need to verify Assumptions (A) and (E). The verification of the RCLT is somewhat involved, since the projection statistic of interest is a {\em weighted} sum with weights depending on the sample size.

\begin{theorem} 
	\label{ThHD1}
	Suppose that (\ref{ConvergenceVarianceProj}) and (\ref{UniformInt}) hold. 
	If $ \tau_a $, $ a > 0 $,  is a sequence of integer-valued random variables with $ \tau_a / a \to k $, in probability, for some $ 0 < k < \infty $, then the statistic $ T_n = \vecw_n'\overline{\vecX}_n $ satisfies the random central limit theorem, i.e.
	\[
	T_{\tau_a} \stackrel{d}{\to} N(0,1),
	\]
	as $ a \to \infty $.
\end{theorem}

Define the estimator
\begin{equation}
\label{DefEstimatior}
\wh{\sigma}_n^2 = \frac{1}{n} \sum_{i=1}^n  \left[ \vecw_n' \vecX_{ni} - \vecw_n' \overline{\vecX}_{n}  \right]^2.
\end{equation}
The following result verifies that Assumptions (E) holds for this estimator under weak technical conditions.

\begin{theorem}
	\label{ThHD2}
	If (\ref{ConvergenceVarianceProj}), then
	\begin{itemize}
		\item[(i)] $ \wh{\sigma}_n^2 - \vecw_n' \bfSigma_n \vecw_n \stackrel{P}{\to} 0 $ and $ \wh{\sigma}_n^2 \stackrel{P}{\to}  \sigma^2_{\vecw} $,
		as $ n \to \infty $.
		\item[(ii)] $ E( \wh{\sigma}_n^2 ) - \vecw_n' \bfSigma_n \vecw_n  \to 0 $ and $ E( \wh{\sigma}_n^2 ) \to  \sigma_\vecw^2  $, as $ n \to \infty $.
	\end{itemize}
\end{theorem}

The above two theorems imply that the proposed two-stage procedure is asymptotically consistent as well as first order and second order efficient. 

\subsection{Proofs}

\subsubsection{Preliminaries on the random central limit theorem}

According to the general theoretical results of the previous section, it suffices to establish Assumptions (A), the validity of the RCLT, and (E) for the statistic of interest, i.e. the projected data. As a preparation, let us briefly review Anscombe's RCLT and sufficient conditions. Consider a sequence $ X_1, X_2, \dots $ of i.i.d. random variables with mean $ \mu $ and finite variance $ \sigma^2 > 0 $ and a family  $ \tau_a $, $ a > 0 $, of integer-valued random variables, often but not necessarily stopping times, with $ \tau_a/ a \to c $, as $ a \to \infty $, for a finite constant $c$. The RCLT asserts that
\[
S_{\tau_a}^* = \frac{ S_{\tau_a} - \tau_a \mu }{ \sqrt{\tau_a} \sigma }  \stackrel{d}{\to} N( 0, 1 ),
\]
as $ a \to \infty $, where for $ n \in \N $
\[
S_n = \sum_{i=1}^n X_i, \qquad S_n^* = \frac{S_n - n \mu}{\sqrt{n} \sigma}.
\]
This means, the sample size $n$ can be replaced by $ \tau_a $.
The basic idea why this holds is as follows, cf. \cite{GoshMukhoSen1997}: The approximation $ \tau_a \approx a c $ suggests that $ S_{\tau_a}^* $ should have the same limiting distribution as $ S_{n_a}^* $ where $ n_a = \lfloor a c \rfloor $. The  CLT gives $ S_{n_a}^* \to N(0,1) $, as $ a \to \infty $, in distribution. Now, for any $ \varepsilon > 0 $ and $ \delta > 0 $, 
\begin{equation}
\label{RCLTwhy}
P( | S_{\tau_a}^* - S_{n_a}^* | \ge \varepsilon ) \le P\left( \left| \frac{\tau_a}{n_a} - 1 \right| > \delta  \right) + 
P\left( \max_{|n-n_a| \le \delta n_a} | S_n^* - S_{n_a}^* |  \ge \varepsilon \right),
\end{equation}
since on the event $ | \tau_a / n_a - 1 | \le \delta $ we have $ | S_{\tau_a}^* - S_{n_a}^* | \le \max_{|n-n_a| \le \delta n_a} | S_n^* - S_{n_a}^* | $. The second term on the right-hand side of (\ref{RCLTwhy}) can be made arbitrarily small, if the sequence $ \{ S_n^* : n \ge 1 \} $ is {\em uniformly continuous in probability} (u.c.i.p). In general, a sequence $ \{ Y_n : n \ge 1 \} $ is called u.c.i.p., if  for any $ \varepsilon > 0 $ there exists some $ \delta > 0 $, so that for large enough $n$
\[
P\left( \max_{0 \le k \le n \delta} | Y_{n+k} - Y_n |  \ge \varepsilon \right) < \varepsilon.
\]
The u.c.i.p property of $ S_n^* $ can be shown using Kolmogorov's maximal inequality, since $ S_n^* $ is a sum of i.i.d. random variables.

\subsection{Proofs}

Let us first show the sufficiency of (\ref{Moment_X}) for (\ref{UniformInt_X}) observe that
\[
E( \| \vecX_{n1} \|_2^r \eins( \| \vecX_{n1} \|_2^k > c ) ) = E\left( \| \vecX_{n1} \|_2^{r+\delta} \| \vecX_n \|_2^{-\delta} \eins( \| \vecX_n \|_2^{-\delta} < c^{1/qk} ) \right)
\le c^{1/qk} E \| \vecX_n \|_2^{r+\delta}.
\]

As a preparation for the following proofs, observe that we may choose $c$ large enough to ensure that
\begin{equation}
  E( \vecw_n'\vecX_{n1} )^2 \le c^2 + \sup_{n \ge 1} E( (\vecw_n'\vecX_{n1})^2 \eins( | \vecw_n'\vecX_{n1} | > c ) ) \le 2 c^2
\end{equation}
holds for all $n \ge 1$. Therefore the sequence of second moments of $ \vecw_n'\vecX_{n1} $ is bounded. This implies
\begin{align}
\label{VarianceBounded}
  \vecw_n'\bfSigma_n \vecw_n &= \Var( \vecw_n'\vecX_{n1} ) = O(1), \\
\label{MeansBounded}
  E | \vecw_n' \vecX_{n1} | & = O(1). 
\end{align}

For technical reasons, we first show Theorem~\ref{ThHD2}.

\begin{proof}[Proof of Theorem~\ref{ThHD2}]
Observe that the random variables $ Z_{ni} = ( \vecw_n' \vecX_{ni} )^2 $, $ i = 1, \dots, n $, are independent and nonnegative with mean $ E(Z_{n1} ) = \vecw_n' \vecM_n \vecw_n $, where $ \vecM_{n1} = E( \vecX_{n1} \vecX_{n1}' ) $. Hence, by the strong law of large numbers for arrays due to \cite{Gut1992} under the uniform integrability condition (e) given there, 
	\[
	   \sup_{n \ge 1} E( (\vecw_n' \vecX_{n1} )^2 \eins( (\vecw_n'\vecX_{n1} )^2 > c ) ) \to 0, \qquad c \to \infty,
	\]
	we obtain
	\begin{equation}
	\label{ConvSecondMoments}
	\frac{1}{n} \sum_{i=1}^n [ (\vecw_n'\vecX_{n1} )^2 - E  (\vecw_n'\vecX_{n1} )^2  ] \stackrel{P}{\to} 0,
	\end{equation}
	as $ n \to \infty $. Similarly, by (\ref{UniformInt}), we also obtain the convergence of the sample moment
	\[
	\frac{1}{n} \sum_{i=1}^n \vecw_n' \vecX_{ni} - \vecw_n' \bfmu \stackrel{P}{\to} 0,
	\]
	as $ n \to \infty $. Further
	\[
	E \left( \frac{1}{n} \sum_{i=1}^n  (\vecw_n'\vecX_{ni} )^2  \right) = E( \vecw_n'\vecX_{n1} \vecX_{n1}' \vecw_n ) =
	\vecw_n' \vecM_{n1} \vecw_n,
	\]
	as $ n \to \infty $,  and
	\begin{align*}
	E ( \vecw_n'\overline{\vecX}_n  )^2 & = \vecw_n' E( \overline{\vecX}_n \overline{\vecX}_n ' ) \vecw_n \\
	& = 
	\frac{1}{n^2} \sum_{i=1}^n \sum_{j=1}^n \vecw_n' E( \vecX_{ni} \vecX_{nj}' ) \vecw_n \\
	& = \frac{1}{n^2} \left[ n \vecw_n' \bfSigma_n\vecw_n + \sum_{1\le i\not=j \le n } \vecw_n' \bfmu_n \bfmu_n' \vecw_n \right],
	\end{align*}
	such that
	\begin{equation}
	\label{ConvSquares}
	E ( \vecw_n'\overline{\vecX}_n  )^2 - ( \vecw_n'\bfmu_n )^2 \to 0,
	\end{equation}
	as $ n \to \infty $, by (\ref{VarianceBounded}). Using the decomposition
	$ \Var( \vecw_n' \vecX_{n1} ) = \vecw_n' \vecM_{n1} \vecw_n -  (\vecw_n'\bfmu_n)^2 $ and rearranging terms,  (\ref{ConvSecondMoments}) and (\ref{ConvSquares}) imply
	\begin{align*}
	\left| E( \wh{\sigma}_n^2  ) - \Var( \vecw_n' \vecX_{n1} ) \right| &= 
	\left| E( \wh{\sigma}_n^2 ) - [ \vecw_n' \vecM_{n1} \vecw_n - (\vecw_n'\bfmu_n)^2 ] \right| \\ & \quad \le \left| E \left( \frac{1}{n} \sum_{i=1}^n   (\vecw_n'\vecX_{ni} )^2 \right) - \vecw_n' \vecM_{n1} \vecw_n \right| + \left|  E ( \vecw_n'\overline{\vecX}_n  )^2 - ( \vecw_n'\bfmu_n )^2   \right| \\
	& \quad \to 0,
	\end{align*}
	as $ n \to \infty $, which verifies (ii).
\end{proof}

\begin{proof}[Proof of Theorem~\ref{ThHD1}] Let $ S_n = \sum_{i=1}^n \vecw_n' \vecX_{ni} $ and observe that the standardized version is given by
	\[
	S_n^* = \frac{ S_n - n \vecw_n' \bfmu }{ \sqrt{n \vecw_n' \bfSigma_n \vecw_n} } = \frac{1}{\sqrt{n}} \sum_{i=1}^n  \xi_{ni}
	\]
	where $ \xi_{ni} = \frac{\vecw_n'\vecX_{ni} - \vecw_n'\bfmu}{ \sqrt{\vecw_n' \bfSigma_n \vecw_n}  } $, $ i = 1, \dots, n $. The random variables $ \xi_{n1}, \dots, \xi_{nn} $ are i.i.d. with $ E( \xi_{n1} ) = 0 $ and $ \Var( \xi_{n1} ) = 1 $ for each $n$. Let us verify that condition (\ref{UniformInt}) provides the Lindeberg condition for the array $ \{ \xi_{ni} : 1 \le i \le n, n \ge 1 \} $. Put
	\[
	  s_n^2 = \Var( S_n^* ) = \frac{1}{n} \sum_{i=1}^n E( \xi_{ni}^2 ).
	\]
	Let $ \delta > 0 $.	First observe that, since the $ \xi_{ni} $ are row-wise identically distributed, the Lindeberg condition collapses to
	\begin{align*}
	  L_n( \delta) &= \frac{1}{s_n^2} \sum_{i=1}^n E\left( \xi_{ni}^2 \eins( | \xi_{ni} | > \delta s_n \sqrt{n} )  \right) \\
	  & = \frac{1}{E(\xi_{n1})^2} E\left( \xi_{n1}^2 \eins( | \xi_{n1} | > \delta s_n \sqrt{n} )  \right).
	\end{align*}
	The latter expression can be rewritten as
	\[
	  \frac{1}{E((\vecw_n'(\vecX_{n1}-\bfmu_n))^2 ) } E\left( (\vecw_n'(\vecX_{n1} - \bfmu_n) )^2 \eins( | \vecw_n'(\vecX_{n1} - \bfmu_n) | > \delta s_n \sqrt{n}  )  \right)
	\]
	The first factor is $ O(1) $ by assumption (\ref{ConvergenceVarianceProj}). Therefore it suffices to show that the second factor, denoted by $ K_n(\delta) $ in what follows, converges to $0$ if $ n \to \infty $. Observe that
	\[
	  K_n(\delta) = R_{n1} - 2 R_{n2} + R_{n3},
	\]
	where
	\begin{align*}
	  R_{n1} &= E\left( (\vecw_n'\vecX_{n1} )^2 \eins( | \vecw_n'(\vecX_{n1}-\bfmu_n) | > \delta s_n \sqrt{n} )  \right) , \\
	  R_{n2} &= E\left( \vecw_n'\vecX_{n1} \vecw_n'\bfmu_n \eins( |\vecw_n'(\vecX_{n1}-\bfmu_n)| > \delta s_n \sqrt{n} ) ) \right), \\
	  R_{n3} & = E\left( (\vecw_n'\bfmu_{n1})^2 \eins( |\vecw_n'(\vecX_{n1}-\bfmu_n)| > \delta s_n \sqrt{n} ) )  \right).
	\end{align*}
	First, observe that
	\begin{align*}
	  |R_{n1}| &\le E \left( (\vecw_n'\vecX_{n1})^2 \eins( | \vecw_n'\vecX_{n1} | + | \vecw_n'\bfmu_n| > \delta s_n \sqrt{n} )  \right) \\
	  & = E \left( (\vecw_n'\vecX_{n1})^2 \eins( | \vecw_n'\vecX_{n1} | > \delta s_n \sqrt{n}  - | \vecw_n'\bfmu_n| )  \right) \\
	  & \to 0,
	\end{align*}
	as $ n \to \infty $, by (\ref{UniformInt}), since $ | \vecw_n' \bfmu_n | = O(1) $, see (\ref{MeansBounded}). Similarly,
	\begin{align*}
	  |R_{n2}| & \le | \vecw_n'\bfmu_n | E\left( |\vecw_n'\vecX_{n1}| \eins( | \vecw_n'\vecX_{n1} | > \delta s_n \sqrt{n}  - | \vecw_n'\bfmu_n| )  \right)  \\
	  & \to 0,
	\end{align*}
	as $ n \to \infty $, again by (\ref{UniformInt}). Lastly, by the Cauchy-Schwarz inequality and (\ref{MeansBounded}),
	\begin{align*}
	  | R_{n3} | & \le ( \vecw_n'\bfmu_n)^2 P( | \vecw_n'(\vecX_{n1}-\bfmu_n) | > \delta s_n \sqrt{n} ) \\
	  & =  O(1) \frac{E( \vecw_n'(\vecX_{n1}-\bfmu_n))^2}{ \delta^2 s_n^2 n } \\
	  & = O\left( \frac{\vecw_n'\bfSigma_n\vecw_n}{\delta^2 n} \right) \\
	  & = o(1),
	\end{align*}
	as $ n \to \infty $, using (\ref{VarianceBounded}) and $ s_n^2 = 1 $.
	Hence, by virtue of the CLT for row-wise i.i.d. arrrays under the Lindeberg condition, see e.g. \cite{Durrett2019}, $S_n^* $ converges in distribution to a standard normal distribution. To verify that the RCLT holds, we need to check the u.c.i.p. condition, see the discussion in the next section for more details. But since the Kolmogorov maximal inequality also applies to the rows  of the row-wise i.i.d. array $ \xi_{n1}, \dots, \xi_{nn} $ depending on $n$, we have
	\[
	P\left( \max_{0 \le k \le n \delta} | S_{n+k}^* - S_n^*| > \frac{\varepsilon}{2} \right) \le 4 (n \varepsilon^2)^{-1} \Var\left( \sum_{i=1}^{n\delta} \xi_{ni} \right) \le 4 \delta / \varepsilon^2,
	\]
	since $\Var\left( \sum_{i=1}^{n\delta} \xi_{ni} \right)  = O( \lfloor n \delta \rfloor ) $. 
	Therefore, the proof of the random central limit theorem can be completed as in \cite[Theorem~2.7.2]{GoshMukhoSen1997}.
\end{proof}

\section{Common mean estimation}
\label{Sec: CommonMean}

\subsection{A review of common mean estimation}

We study the construction of a fixed-width confidence interval for the common mean of two samples based on two-stage sequential estimation with equal sample sizes. For motivation consider the following example: A good is produced using two machines operating at the same speed but with possibly different accuracies. Both machines have to be used to satisfy demand and interest focuses on estimation of the common mean and providing a fixed-width confidence interval for it. The same situation arises when combining data from two laboratories where usually one laboratory has better equipment or expertise than the other lab and therefore provides the measurements with smaller uncertainty, such that the assumption of ordered variances is justified. In such settings, one should base inference on two random samples with equal sample sizes, and the goal is to determine the minmal number of observations required to estimate the common mean with preassigned accuracy.

Let us briefly review some related facts and results about common mean estimation. The literature focuses on parametric settings, mainly the case of two independent Gaussian samples. It is, nevertheless, worth recalling the estimators studied under normality. Indeed, our aim is to study a class of common mean estimators which covers many of the estimators proposed in the literature as special cases. 

It is well known that for Gaussian samples the minimal variance unbiased estimator in the case of known variances is a convex combination of the sample means $ \overline{X}_i = N^{-1} \sum_{j=1}^{N} X_{ij} $, $ i = 1, 2 $, with weights depending on the unknown $ \sigma_1^2 $ and $ \sigma_2^2 $. Hence, it is natural to study weighted means of the sample means where the weights depend on the sample variances $ S_i^2 =  N^{-1} \sum_{j=1}^{N} (X_{ij} - \overline{X}_i )^2 $, $ i = 1, 2 $, as we shall do here, or equivalently depending on their unbiased versions denoted by $ \wt{S}_i^2 $, $ i = 1, 2 $, throughout the paper. The canoncial approach is to estimate the optimal weights by substituting the unkown variances by their unbiased estimators. First proposed and studied by \cite{GraybillDeal1959}, this leads to the {\em Graybill-Deal (GD)} estimator
\[
\wh{\mu}^{(GD)} = \frac{ \wt{S}_2^2 }{  \wt{S}_2^2 +  \wt{S}_1^2 } 
\overline{X}_1 + \frac{  \wt{S}_1^2 }{ \wt{S}_2^2 + \wt{S}_1^2 } 
\overline{X}_2,
\]
which does assume an order constraint. Observe that the GD estimator is a convex combination of the sample means with random weights given by smooth functions the unbiased variance estimators. When there is an order constraint on the variances, $ \sigma_1^2 \le \sigma_2^2 $, the GD estimator can be improved and several proposal have been made and investigated, see the discussion in \cite{StelandChang2019}. For example,
 \cite{Nair1982} showed that the GD estimator can be improved by using random convex weights whose definition (as smooth functions of the sample variances) depends on the ordering of variance estimates: If $ \wt{S}_1^2 > \wt{S}_2^2 $, then
Nair's proposal switches from the formula for $ \wh{\mu}^{(GD)} $ to $ \frac{1}{2} \overline{X}_1 + \frac{1}{2} \overline{X}_2 $. This estimator stochastically dominates the GD estimator, as shown by \cite{Elfessi1992}. A further proposal is to switch to the formula $ \frac{\wt{S}_1^2}{\wt{S}_1^2+\wt{S}_2^2} \overline{X}_1 + \frac{\wt{S}_2^2}{\wt{S}_1^2+\wt{S}_2^2} \overline{X}_2 $ instead, see
\cite{Elfessi1992}. As shown in \cite{ChangEtAl2012}, such an estimator can be further improved in terms of stochastic dominance by further modifying the weights, see also \cite{ChangShinozaki2015} for a study using a different criterion.

The above findings motivate to study the general class of common mean estimators given by
\begin{equation}
	\label{ConsideredClass}
	\wh{\mu}_N( \gamma ) = \gamma_N \overline{X}_1 + (1-\gamma_N) \overline{X}_2,
\end{equation}
where the weight $ \gamma_N $ is random and of the form
\begin{equation}
\label{SpecificWeights}
\gamma_N = \gamma( \wt{S}_1^2, \wt{S}_2^2, \overline{X}_1, \overline{X}_2 )
= \left\{
\begin{array}{ll}
\gamma^{\le}( \wt{S}_1^2, \wt{S}_2^2, \overline{X}_1, \overline{X}_2 ), \qquad \wt{S}_1^2 \le \wt{S}_2^2, \\
\gamma^>( \wt{S}_1^2, \wt{S}_2^2, \overline{X}_1, \overline{X}_2 ), \qquad 
\wt{S}_1^2 > \wt{S}_2^2. \\
\end{array}  
\right.
\end{equation}
Here the functions $ \gamma^\le $ and $ \gamma^> $ are three times continuously differentiable with bounded derivatives.

Although the existing literature suggests to use such convex combinations with random weights due to their performance in terms of variance, squared error losses as in \cite{CohenSackrowtz1974} or stochastic dominance, the established methodology to use those estimators for inference is rather limited and is mainly confined to Gaussian samples. Even estimation of the variance of these estimators is non-trivial and not well studied.

In the sequel, we relax the restrictive condition of Gaussian measurements and consider estimators of the the above class for indendent samples following a distribution with finite $ 12$th moment. Under this condition,  \cite{StelandChang2019} have shown shown that  jackknife variance estimators are consistent and asymptotically normal for common mean estimators from the class (\ref{ConsideredClass}). For the special case of equal sample sizes as studied here a simplified jackknifing scheme was proposed. 
Especially, it follows from \cite{StelandChang2019} that 
\[
  \sqrt{N}( \wh{\mu}_N(\gamma) - \mu ) \stackrel{d}{\to} N(0, \sigma_\mu^2(\gamma) ),
\]
with asymptotic variance
\[
  \sigma_\mu^2 =  2 [ \gamma^2 \sigma_1^2 + (1-\gamma)^2 \sigma_2^2 ]
\]
and $ \gamma = \gamma( \sigma_1^2 , \sigma_2^2, \mu, \mu ) $. When combined with the validity of the RCLT established in the next section, that result allows the construction of a fixed-width confidence interval for the common mean using a two-stage sampling approach, according to the results of the previous section. Therefore, we establish the RCLT for the above class of common mean estimators. 

\subsection{The random central limit theorem for common mean estimators}

For the class of common mean estimators under investigation the u.c.i.p. property can not be shown by a simple application of Kolmogorov's inequality. The situation is more involved and requires special treatment, since the weights are random and depend on the sample size. 

The following three main results on common mean estimators with random weights assert that the RCLT holds. It turns out that the RCLT is directly related to the tightness of the weighting sequence. This is formulated in the first theorem: For {\em arbitrary random} weights $ \gamma_N $ which are bounded and u.c.i.p. the RCLT holds.

\begin{theorem}
	\label{RCLTGeneral}
	Let $ X_{ij} $, $ j = 1, \dots, n_i $, $ i = 1,2 $, be i.i.d. samples of  random variables with finiite second moment. 
	Suppose that $ \{ \gamma_n : n \ge 1 \} $ is a sequence of bounded random weights satisfying the u.c.i.p. condition. Let $ \tau_a $, $ a > 0 $, be a sequence of integer-valued random variables with $ \tau_a/ a \to c $, as $ a \to \infty $. If $ \sqrt{N}( \wh{\mu}_N(\gamma) - \mu ) / \sigma_\mu $  satisfies the CLT, then 
	$ \sqrt{N}( \wh{\mu}_N(\gamma) - \mu ) / \sigma_\mu $ also satisfies the RCLT, i.e.
	\[
	\sqrt{\tau_a}( \wh{\mu}_{\tau_a}( \gamma ) - \mu ) / \sigma_\mu \stackrel{d}{\to} N(0,1), 
	\]
	as $ a \to \infty $.
\end{theorem}

The next result shows that the  weights $ \gamma_N $ considered by the class of common mean estimators (\ref{ConsideredClass}) satisfy the u.c.i.p. property. 

\begin{theorem} 
	\label{WCM_ucip}
	The weights $ \gamma_N $ (\ref{SpecificWeights}) of the common mean estimator have the u.c.i.p. property if $ E (X_{i1}^4) < \infty $, $ i = 1, 2$.
\end{theorem}

Consequently, we arrive at the following RCLT for the class of common mean estimators of interest, which verifies Assumption (A) for them.

\begin{theorem}
	\label{RCLTCommonMean}
	Let $ X_{ij} $, $ j = 1, \dots, n_i $, $ i = 1,2 $, be i.i.d. samples of  random variables with $ E | X_{i1} |^{12} < \infty $. Let $ \tau_a $, $ a > 0 $, be a family of integer-valued random variables with $ \tau_a/ a \to c $, as $ a \to \infty $. Then the common mean estimator $ \sqrt{N}( \wh{\mu}_N(\gamma) - \mu ) / \sigma_\mu $ satisfies the RCLT, i.e.
	\[
	\sqrt{\tau_a}( \wh{\mu}_{\tau_a}( \gamma ) - \mu ) / \sigma_\mu \stackrel{d}{\to} N(0,1), 
	\]
	as $ a \to \infty $, such that Assumption (A) holds true.
\end{theorem}

\subsection{Proofs}

\begin{proof}[Proof of Theorem~\ref{RCLTGeneral}]
	We have the representation
	\[
	\wh{\mu}_N( \gamma ) - \mu = \frac{1}{N} \left\{  \gamma_N S_N^{(1)} + (1-\gamma_N) S_N^{(2)}  \right\},
	\]
	where $ S_k^{(i)} = \sum_{j=1}^k (X_{ij}-\mu) $, $ k \ge 1 $, for $ i = 1, 2 $. Hence,
	\[
	\sqrt{N}( \wh{\mu}_N( \gamma ) - \mu ) = \gamma_N S_N^{(1)*} + (1-\gamma_N) S_N^{(2)*} 
	\]
	We may assume $ \mu = 0 $ and $ \sigma_\mu = 1 $. Then $ S_n^* = n^{-1/2} S_n $ and for the sequence of the standardized versions $ \wh{\mu}_n^*( \gamma ) = \sqrt{n} \wh{\mu}_n $ we obtain
	\begin{align*}
	\wh{\mu}_{n+k}^*( \gamma )  - \wh{\mu}_n^*( \gamma ) &= \gamma_{n+k} S_{n+k}^{(1)*} - \gamma_n S_n^{(1)*} + (1-\gamma_{n+k}) S_{n+k}^{(2)*} - (1-\gamma_n) S_n^{(2)*} \\
	& = \gamma_{n+k}( S_{n+k}^{(1)*} - S_{n}^{(1)*} ) + ( \gamma_{n+k} - \gamma_n) S_n^{(1)*}  \\
	& \quad + (1-\gamma_{n+k}) ( S_{n+k}^{(2)*} - S_{n}^{(2)*} ) + ( \gamma_{n} - \gamma_{n+k}) S_n^{(2)*}
	\end{align*}
	By boundedness of the weights, we may assume the bound is $1$, we may find $ \delta > 0 $ such that 
	\[
	P\left( \max_{0 \le k \le n \delta} |  \gamma_{n+k}( S_{n+k}^{(1)*} - S_{n}^{(1)*} ) | \ge \varepsilon   \right) 
	\le P\left( \max_{0 \le k \le n \delta} | S_{n+k}^{(1)*} - S_{n}^{(1)*} |  \ge \varepsilon/4   \right) < \varepsilon/4
	\]
	for all $ n \ge 1 $. For the third term the same argument applies. The remaining two terms require the u.c.i.p. property of the weights and are treated as follows: We may find a constant $ C > 8 $ with
	$P(|  S_n^{(1)*} | > C ) < \varepsilon/8 $, $ n \ge 1 $. Now choose $ \delta > 0 $ so that
	\[
	P\left( \max_{0 \le k \le n \delta} | \gamma_{n+k} - \gamma_n | \ge \varepsilon/C \right) < \varepsilon/C \le \varepsilon/8.
	\]
	This leads to
	\begin{align*}
	P\left( \max_{0 \le k \le n \delta} |  \gamma_{n+k} - \gamma_n | | S_{n}^{(1)*} |  \ge \varepsilon   \right) 
	& \le 
	P\left( \max_{0 \le k \le n \delta} C |  \gamma_{n+k} - \gamma_n | \ge \varepsilon,  |  S_n^{(1)*} | \le C  \right) + P(|  S_n^{(1)*} | > C ) \\
	& \le P\left( \max_{0 \le k \le n \delta} |  \gamma_{n+k} - \gamma_n | \ge \varepsilon/C  \right) + \frac{\varepsilon}{8} \\
	&\le \frac{\varepsilon}{4}.
	\end{align*}
	The remaining fourth term is treated analogously. It follows that the sequence of standardized estimates $  \wh{\mu}_n^* $, $ n \ge 1 $, is u.c.i.p., i.e.
	\[
	P\left( \max_{0 \le k \le n \delta} | \wh{\mu}_{n+k}^*( \gamma )  - \wh{\mu}_n^*( \gamma ) | \ge \varepsilon \right) < \varepsilon, \qquad n \ge 1.
	\]
	Since $  \wh{\mu}_n^* \stackrel{d}{\to} N(0,1) $, as $ n \to \infty $, the RCLT follows, cf. (\ref{RCLTwhy}).
\end{proof}

\begin{proof}[Proof of Theorem~\ref{WCM_ucip}]
	In what follows we need to make the sample sizes explicit in the notation for the estimators and therefore write
	$ \wt{S}_{n,i} $, $ \overline{X}_{n,i} $ for $ i = 1,2 $. Of course, $ \sqrt{n} \overline{X}_{n,i} $, $n \ge 1 $, and therefore $  \overline{X}_{n,i} $, $ n \ge 1 $, is u.c.i.p.. Similarly, $ V_{n,i} = n^{-1} \sum_{j=1}^n X_{ij}^2 $, $ n \ge 1 $, is u.c.i.p.. Next observe that
	\begin{align*}
	\wt{S}_{n+k,i}^2 - \wt{S}_{n,i}^2 & = \frac{1}{n+k} \sum_{j=1}^{n+k} X_{ij}^2 - \frac{1}{n} \sum_{j=1}^n X_{ij}^2 - ( \overline{X}_{n+k,i} )^2 +  ( \overline{X}_{n,i} )^2 \\
	& = V_{n+k,i} - V_{n,i} - (  \overline{X}_{n+k,i}  -  \overline{X}_{n,i}   ) ( \overline{X}_{n+k,i}  +
	\overline{X}_{n,i} )  
	\end{align*}
	We may find $ C > 4 $ with $ P( | \overline{X}_{n+k,i}  +
	\overline{X}_{n,i} | >  C ) < \varepsilon/4 $.  Now choose $ \delta > 0 $ such that
	\[
	P\left( \max_{0 \le k \le n \delta} | V_{n+k,i} - V_{n,i} | \ge \frac{\varepsilon}{2} \right) < \frac{\varepsilon}{2}, \qquad i = 1, 2, \dots, \quad n \ge 1,
	\]
	as well as
	\[
	P\left( \max_{0 \le k \le n \delta} |  \overline{X}_{n+k,i}  -  \overline{X}_{n,i}   | \ge \frac{\varepsilon}{C} \right) < \frac{\varepsilon}{C}, \qquad i = 1, 2,  \dots, \quad n \ge 1.
	\]
	Then,
	\begin{align*}
	P\left( \max_{0 \le k \le n \delta} | \wt{S}_{n+k,i}^2 - \wt{S}_{n,i}^2 |  \ge \varepsilon \right)
	& \le P\left( \max_{0 \le k \le n \delta} | V_{n+k,i} - V_{n,i} | \ge \frac{\varepsilon}{2} \right) \\
	& \qquad + P\left( \max_{0 \le k \le n \delta} |  \overline{X}_{n+k,i}  -  \overline{X}_{n,i}   | | \overline{X}_{n+k,i}  +
	\overline{X}_{n,i} | \ge \frac{\varepsilon}{2} \right) \\
	& \le \frac{\varepsilon}{2} + P\left( \max_{0 \le k \le n \delta} |  \overline{X}_{n+k,i}  -  \overline{X}_{n,i}   | \ge \frac{\varepsilon}{C} \right) +  P( | \overline{X}_{n+k,i}  +
	\overline{X}_{n,i} | >  C ) \\
	& < \frac{\varepsilon}{2} +  \frac{\varepsilon}{4}  + \frac{\varepsilon}{4},
	\end{align*}
	which verifies the u.c.i.p. property for $ \wt{S}_{n,i} $, $ n \ge 1 $. For a function $ f $ of several variables write $ \partial_j f $ for the partial derivative with respect to the $j$th argument. Observe that on $ \{ \wt{S}_{n,1} \le \wt{S}_{n,2} \} $ a Lipschitz constant of $ \gamma $ with respect to the $j$th variable is given by $ \| \partial_j \gamma^{\le} \|_\infty $, and on $ \{ \wt{S}_{n,1} > \wt{S}_{n,2} \} $ we have the Lipschitz constant $ \| \partial_j \gamma^{>} \|_\infty $, $ j = 1, \dots, 4 $. Therefore, a Lipschitz constant on $ \Omega $ is given by $ \max \{ \| \partial_j \gamma^{\le} \|_\infty, \| \partial_j \gamma^{>} \|_\infty \} $. We obtain the  following Lipschitz property:
	\begin{align*}
	|\gamma_{n+k} - \gamma_n| & = |\gamma(  \wt{S}_{n+k,1}^2, \wt{S}_{n+k,2}^2, \overline{X}_{n+k,1}, \overline{X}_{n+k,2} ) -
	\gamma(  \wt{S}_{n,1}^2, \wt{S}_{n,2}^2, \overline{X}_{n,1}, \overline{X}_{n,2} ) | \\
	& \le \max\{ \| \partial_1 \gamma^{\le}  \|_\infty, \| \partial_1 \gamma^> \|_\infty \} | | \wt{S}_{n+k,1}^2 - \wt{S}_{n,1}^2| \\
	& \quad + \max\{ \| \partial_2 \gamma^{\le}  \|_\infty, \| \partial_2 \gamma^> \|_\infty \} |  \wt{S}_{n+k,2}^2 - \wt{S}_{n,2}^2|  \\
	& \quad + \max\{ \| \partial_3 \gamma^{\le}  \|_\infty, \| \partial_3 \gamma^> \|_\infty \} | \overline{X}_{n+k,1} - \overline{X}_{n,1}|  \\
	& \quad + \max\{ \| \partial_4 \gamma^{\le}  \|_\infty, \| \partial_4 \gamma^> \|_\infty \} | \overline{X}_{n+k,2} - \overline{X}_{n,2}|. 
	\end{align*}
	Consequently, there exists $ 0 < L < \infty$ with
	\[
	|\gamma_{n+k} - \gamma_n| \le L \left( | \wt{S}_{n+k,1}^2 - \wt{S}_{n,1}^2| + | \wt{S}_{n+k,2}^2 - \wt{S}_{n,2}^2|  +  | \overline{X}_{n+k,1} - \overline{X}_{n,1}|  +  | \overline{X}_{n+k,2} - \overline{X}_{n,2}|  \right),
	\]
	and from this fact the u.c.i.p. property follows easily, since
	\begin{align*}
	P\left( \max_{0 \le k \le n \delta} |\gamma_{n+k} - \gamma_n|  \ge \varepsilon \right)
	&\le \sum_{i=1,2} P\left( \max_{0 \le k \le n \delta} | \wt{S}_{n+k,i}^2 - \wt{S}_{n,i}^2 |  \ge \frac{\varepsilon}{4L} \right) \\
	& \qquad +  \sum_{i=1,2}  P\left( \max_{0 \le k \le n \delta} | \overline{X}_{n+k,i} - \overline{X}_{n,i} |  \ge \frac{\varepsilon}{4L} \right)
	\end{align*}
	and we may find $ \delta > 0 $ such that each term is smaller than $ \varepsilon/4 $.
\end{proof}

\begin{proof}[Proof of Theorem~\ref{RCLTCommonMean}]
	The CLT for $ \wh{\mu}_N(\gamma) $ has been shown under the stated assumptions in \cite{ChoSteland2016}.
	Therefore the results follows from Theorems~\ref{RCLTGeneral} and \ref{WCM_ucip}.
\end{proof}

\section{Simulations and Data Example}
\label{Sec: Simulations}

\subsection{One-sample setting}

The aim of the simulations is to investigate the accuracy of the proposed two-stage procedure when applied to non-normal data. Especially, it is of interest to examine whether the proposed {\em high-accuracy} asymptotic frameworks works. 

For simplicity, we used the arithmetic mean to estimate the mean $ \mu $. I.i.d. data were simulated following the model
\[
  X = \mu +  \epsilon
\]
with $ \mu = 10 $ and zero mean error terms $ \epsilon $ distributed according to a standard normal distribution (model 1), a $ t(5) $-distribution (model 2) showing heavier tails and a $ U(-2.5, 2.5) $ distribution (model 3). The confidence level was chosen as $ 95\% $ and $ 99\% $. Further, the accuracy parameter $d$ was selected from the values $ \{ 0.3, 0.2, 0.1, 0.05 \} $. 

The first-stage  sample size was calculated using the three-observations-rule and a minimal sample size of $ \overline{N}_0 $ equal to $ 15 $ or $ 30 $. To calculate $ \wh{N}_{opt} $ the asymptotic variance of estimator was estimated  by $ \wh{S}_{N_0}^2 = \frac{1}{N_0-1} \sum_{i=1}^{N_0} (X_i - \overline{X}_{N_0} )^2 $ using the first-stage sample.

Table~\ref{simtab1} shows the results. It can be seen that the accuracy of the procedure is very good, both in terms of the coverage probability and in terms of over- or undershooting. The coverage probability is only slightly smaller than the nominal value. This effect is more pronounced for heavier tails, whereas shorter tails somewhat compensate this effect. The over- or undershooting is quite moderate, even for a minimal sample size of $ \overline{N}_0 = 15 $. When the minimal sample size is $ 30 $ instead of $15$, then the over- or undershooting is further reduced.

Fixing a large value for $d$, say $ d = 0.3 $, and comparing the results for the confidence levels $ 90\% $ and $ 99\% $ one can observe that the high-accuracy asymptotics works very well for cases under investigation: A high confidence level ensures even for large, fixed $d$ convincing coverage. 

\begin{table}[ht]
	\centering
	\begin{tabular}{ccccccccc}
		\hline
		& & & \multicolumn{3}{c}{$\overline{N}_0 = 15$} & \multicolumn{3}{c}{$\overline{N}_0 = 30$} \\
		$1-\alpha$ & Model & $d$ & $p$ & $E(\wh{N}_{opt})$ & $ E(\wh{N}_{opt})-N_{opt}^*$ & $p$ & $E(\wh{N}_{opt})$ & $ E(\wh{N}_{opt})-N_{opt}^*$  \\ 
		\hline
		0.95 & 1 & 0.30 & 0.94 & 45.29 & 0.29 & 0.95 & 45.41 & 0.41 \\ 
		0.95 & 1 & 0.20 & 0.93 & 98.35 & -0.65 & 0.94 & 98.59 & -0.41 \\ 
		0.95 & 1 & 0.10 & 0.94 & 388.54 & 1.54 & 0.94 & 387.13 & 0.13 \\ 
		0.95 & 1 & 0.05 & 0.94 & 1540.40 & 1.40 & 0.94 & 1536.66 & -2.34 \\ 
		0.95 & 2 & 0.30 & 0.93 & 73.77 & -0.23 & 0.94 & 73.66 & -0.34 \\ 
		0.95 & 2 & 0.20 & 0.92 & 161.75 & -1.25 & 0.93 & 162.36 & -0.64 \\ 
		0.95 & 2 & 0.10 & 0.92 & 645.93 & 2.93 & 0.93 & 645.01 & 2.01 \\ 
		0.95 & 2 & 0.05 & 0.93 & 2565.71 & 2.71 & 0.93 & 2565.50 & 2.50 \\ 
		0.95 & 3 & 0.30 & 0.94 & 91.22 & 0.22 & 0.95 & 91.52 & 0.52 \\ 
		0.95 & 3 & 0.20 & 0.94 & 202.96 & -0.04 & 0.94 & 202.63 & -0.37 \\ 
		0.95 & 3 & 0.10 & 0.95 & 802.01 & -0.99 & 0.95 & 801.95 & -1.05 \\ 
		0.95 & 3 & 0.05 & 0.95 & 3208.29 & 4.29 & 0.95 & 3203.94 & -0.06 \\ 
		0.99 & 1 & 0.30 & 0.98 & 76.24 & 0.24 & 0.99 & 76.04 & 0.04 \\ 
		0.99 & 1 & 0.20 & 0.98 & 168.48 & 0.48 & 0.99 & 168.52 & 0.52 \\ 
		0.99 & 1 & 0.10 & 0.98 & 665.94 & -0.06 & 0.99 & 666.43 & 0.43 \\ 
		0.99 & 1 & 0.05 & 0.99 & 2666.11 & 10.11 & 0.99 & 2657.20 & 1.20 \\ 
		0.99 & 2 & 0.30 & 0.97 & 125.28 & 0.28 & 0.98 & 125.49 & 0.49 \\ 
		0.99 & 2 & 0.20 & 0.97 & 281.14 & 2.14 & 0.98 & 280.46 & 1.46 \\ 
		0.99 & 2 & 0.10 & 0.98 & 1113.17 & 5.17 & 0.98 & 1108.43 & 0.43 \\ 
		0.99 & 2 & 0.05 & 0.98 & 4399.46 & -26.54 & 0.98 & 4423.36 & -2.64 \\ 
		0.99 & 3 & 0.30 & 0.98 & 155.93 & -0.07 & 0.99 & 156.22 & 0.22 \\ 
		0.99 & 3 & 0.20 & 0.99 & 346.53 & -1.47 & 0.99 & 348.09 & 0.09 \\ 
		0.99 & 3 & 0.10 & 0.99 & 1384.75 & -0.25 & 0.99 & 1383.22 & -1.78 \\ 
		0.99 & 3 & 0.05 & 0.99 & 5534.17 & 2.17 & 0.99 & 5522.13 & -9.87 \\ 
		\hline
	\end{tabular}
	\caption{One-sample setting: Fixed-width interval for the mean: Simulated coverage probabilies ($p$), expected sample sizes and over- or undershooting for the three-observations first-stage rule and $ \overline{N}_0 = 15 $. The left columns show the results for a minimal sample size of $15$, the right part of the table for a mimimum of $30$ observations.}
\label{simtab1}
\end{table}
 
\subsection{Common mean estimation}

In order to investigate the statistical accuracy of the proposed two-stage procedure to construct a fixed width confidence interval for the common mean, a simulation study was conducted. Data for the $i$th sample was simulated according to the model
\[
X_i = \mu_0 + \sigma_i \epsilon,
\]
with $ \mu_0 = 10 $, $ \sigma_1 = 1 $, $ \sigma_2 \in \{ 1, 1/2 \} $ and noise terms $ \epsilon \sim N(0,1) $ (model 1),
$ \epsilon \sim t(5) $ (model 2) to study the effect of heavier tails and $ \epsilon \sim U(-2.5, 2.5 ) $ (model 3) as a distribution with short tails. The confidence level was chosen as $ 95 \% $ and $ 99 \% $. Lastly, the precision parameter $d$ was chosen as $ 0.3, 0.2, 0.1 $ and $ 0.05 $. 

We used at least $ \overline{N}_0 = 10 $ observations for each sample, i.e $ 20 $ in total. From the first-stage sample of size $ N_0 $, i.e. using the three-observation rule, the asymptotic variance of the GD common mean estimator was estimated using the canoncial estimator $ \wh{\sigma}_\mu^2 = 2 [ \gamma_N^2 \wt{S}_1^2+ (1-\gamma_N)^2 \wt{S}_2^2 ] $ with the GD weights given by $ \gamma_N = \frac{\wt{S}_2^2}{\wt{S}_2^2 + \wt{S}_1^2} $, in order to calculate then $ \wh{N}_{opt} $. The results are provided in Table~\ref{simtab2}. One can observe that the approach works well, although in general the true coverage probabilities are somewhat lower than nominal. 

\begin{table}[ht]
	\centering
	\begin{tabular}{ccccccccc}
		\hline
		$1-\alpha$ & Model  & & \multicolumn{3}{c}{$\sigma_1=\sigma_2=1$} &  \multicolumn{3}{c}{$\sigma_1=1, \sigma_2=1/2$}  \\
		& & $d$ & $p$ & $E(\wh{N}_{opt})$ & $E(\wh{N}_{opt}) - N_{opt}^* $ & $p$ & $E(\wh{N}_{opt})$ & $E(\wh{N}_{opt}) - N_{opt}^* $ \\ 
		\hline
		0.95 & 1 & 0.30 & 0.90 & 19.73 & -1.61 & 0.92 & 9.43 & 0.89 \\ 
0.95 & 1 & 0.20 & 0.92 & 44.47 & -3.55 & 0.91 & 19.00 & -0.21 \\ 
0.95 & 1 & 0.10 & 0.93 & 182.28 & -9.79 & 0.93 & 74.61 & -2.22 \\ 
0.95 & 1 & 0.05 & 0.94 & 746.39 & -21.90 & 0.94 & 301.10 & -6.22 \\ 
0.95 & 2 & 0.30 & 0.90 & 29.61 & -5.96 & 0.92 & 13.46 & -0.77 \\ 
0.95 & 2 & 0.20 & 0.91 & 68.28 & -11.75 & 0.92 & 28.76 & -3.25 \\ 
0.95 & 2 & 0.10 & 0.92 & 287.46 & -32.66 & 0.92 & 117.83 & -10.22 \\ 
0.95 & 2 & 0.05 & 0.93 & 1201.86 & -78.63 & 0.93 & 483.77 & -28.42 \\ 
0.95 & 3 & 0.30 & 0.92 & 43.50 & -0.96 & 0.91 & 18.44 & 0.65 \\ 
0.95 & 3 & 0.20 & 0.94 & 97.90 & -2.14 & 0.92 & 40.14 & 0.13 \\ 
0.95 & 3 & 0.10 & 0.94 & 394.37 & -5.78 & 0.94 & 159.16 & -0.90 \\ 
0.95 & 3 & 0.05 & 0.95 & 1588.93 & -11.68 & 0.95 & 636.32 & -3.93 \\ \hline
0.99 & 1 & 0.30 & 0.97 & 34.13 & -2.73 & 0.97 & 14.90 & 0.15 \\ 
0.99 & 1 & 0.20 & 0.98 & 77.41 & -5.53 & 0.97 & 32.28 & -0.89 \\ 
0.99 & 1 & 0.10 & 0.98 & 318.08 & -13.66 & 0.98 & 128.86 & -3.83 \\ 
0.99 & 1 & 0.05 & 0.99 & 1297.50 & -29.48 & 0.98 & 521.04 & -9.75 \\ 
0.99 & 2 & 0.30 & 0.97 & 51.93 & -9.51 & 0.97 & 22.24 & -2.34 \\ 
0.99 & 2 & 0.20 & 0.97 & 120.35 & -17.87 & 0.97 & 49.94 & -5.35 \\ 
0.99 & 2 & 0.10 & 0.98 & 507.36 & -45.55 & 0.98 & 206.32 & -14.85 \\ 
0.99 & 2 & 0.05 & 0.98 & 2097.25 & -114.38 & 0.98 & 844.02 & -40.63 \\ 
0.99 & 3 & 0.30 & 0.98 & 75.10 & -1.70 & 0.97 & 30.98 & 0.26 \\ 
0.99 & 3 & 0.20 & 0.99 & 170.02 & -2.76 & 0.98 & 69.01 & -0.10 \\ 
0.99 & 3 & 0.10 & 0.99 & 683.24 & -7.89 & 0.99 & 275.08 & -1.37 \\ 
0.99 & 3 & 0.05 & 0.99 & 2748.73 & -15.81 & 0.99 & 1101.35 & -4.47 \\ 
		\hline \\
	\end{tabular}
	\caption{Fixed-width interval for the common mean: Simulated coverage probabilies ($p$) and expected sample sizes for the three-observations first-stage rule and $ \overline{N}_0 = 10 $.}
	\label{simtab2}
\end{table}

\subsection{Data example}

In chip manufacturing, the width of the cut out chips is a critical quantity and the machines, which operate with different accuracies, have to be calibrated well, in order to meet the specifications. Consequently, quality samples taken from different machines have a common mean but different variances. For quality control purposes, it is of great interest to be in a position to obtain fixed width confidence intervals, in order to report the chip's width at a specified uncertainty level.  

To design, analyze and assess a fixed width confidence interval, we have data from two cutting machines at our disposal. Those measurements are non-normal, as confirmed by the Shapiro-Wilk test, and the sample autocorrelation functions are in good agreement with the i.i.d. assumption, as can be seen from Figure~\ref{ACFplots}. For the problem at hand, a precision of $ 0.01 $ (i.e. $ d = 0.005 $) for the mean chip width at a confidence level of $ 99\% $ was selected. The precision parameter $d$ is rougly equal to one eighth of the measurement's standard deviation. 

\begin{figure}
	\begin{center}
		\includegraphics[width=10cm]{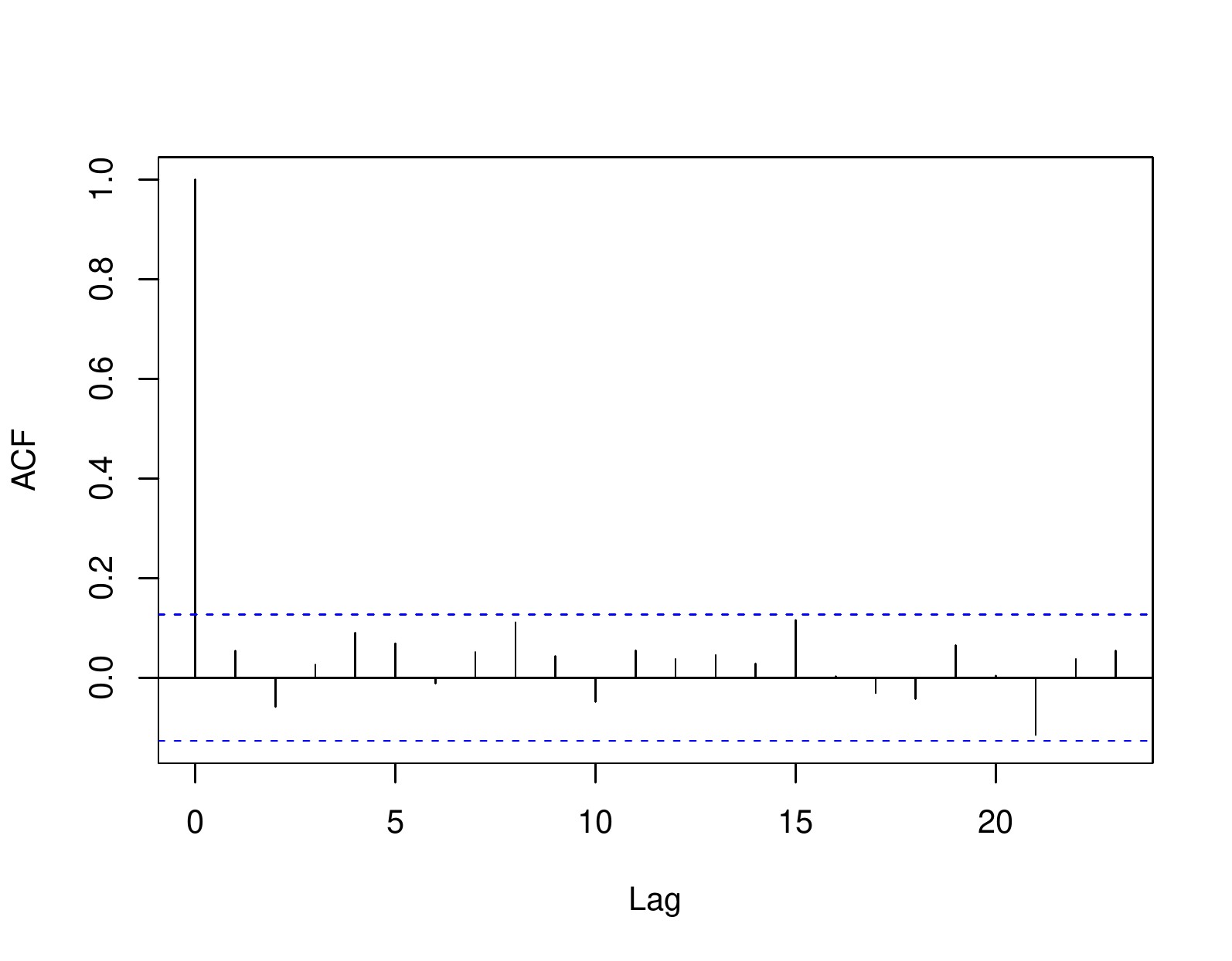}
		\includegraphics[width=10cm]{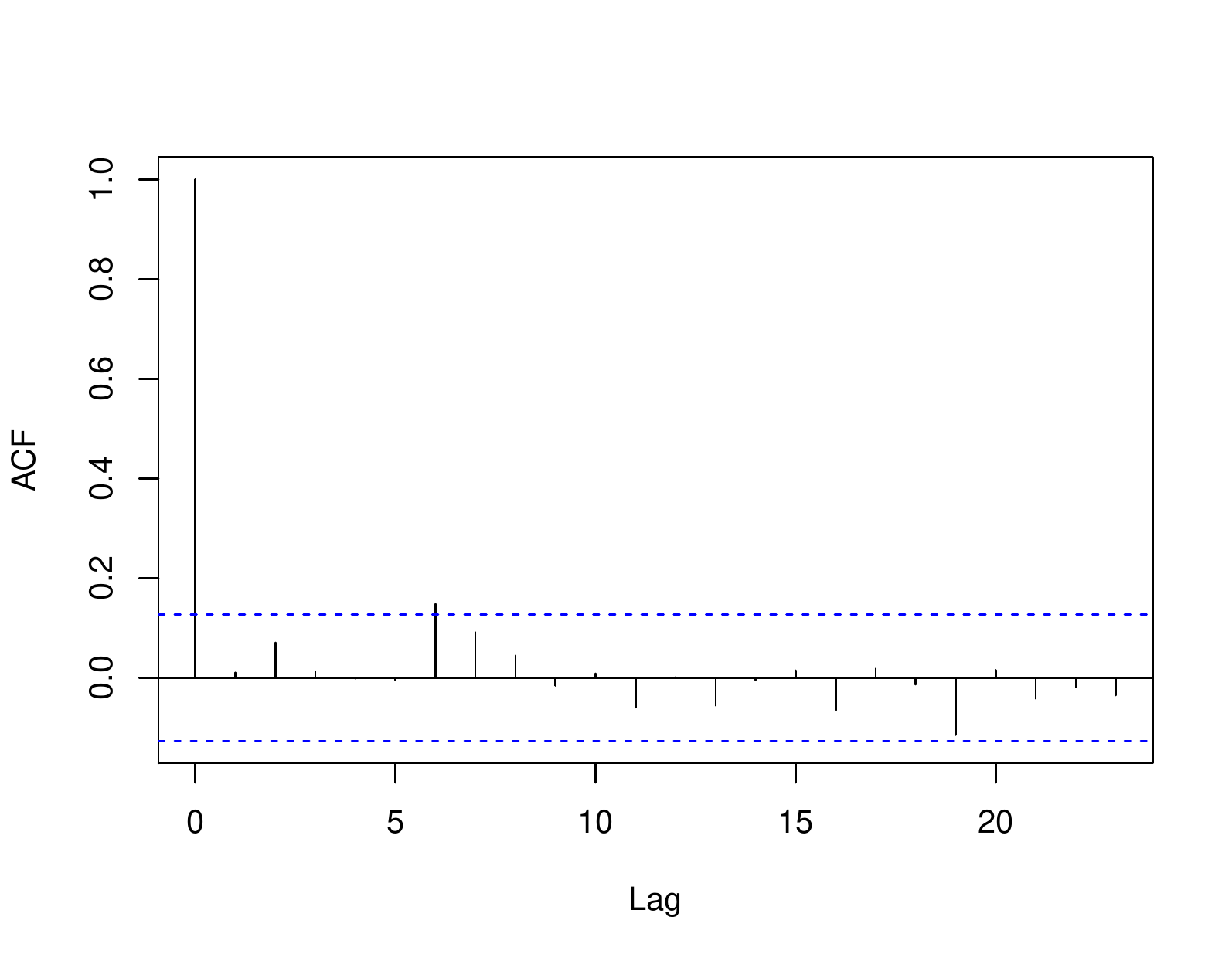}
	\end{center}
	\caption{Autocorrelation functions for the chip width measurements from machine 1 and 2.}
	\label{ACFplots}
\end{figure}

When calculating the classical first-stage sample size, one gets $ N_0 = \lfloor \Phi^{-1}(1-\alpha/2) / d \rfloor + 1 = 5152 $, because the standard deviation is very low but not taken into account by this formula. Instead, the proposed three-observations rule, however, leads to $ N_0 = 23 $, and using the first $ 23 $ observations from both machines leads to $ \wh{N}_{opt} = 145 $. The resulting confidence interval using the GD common mean estimator is given by $ [ 6.292313; 6.293313 ]$. 

\begin{figure}
	\begin{center}
		\includegraphics[width=10cm]{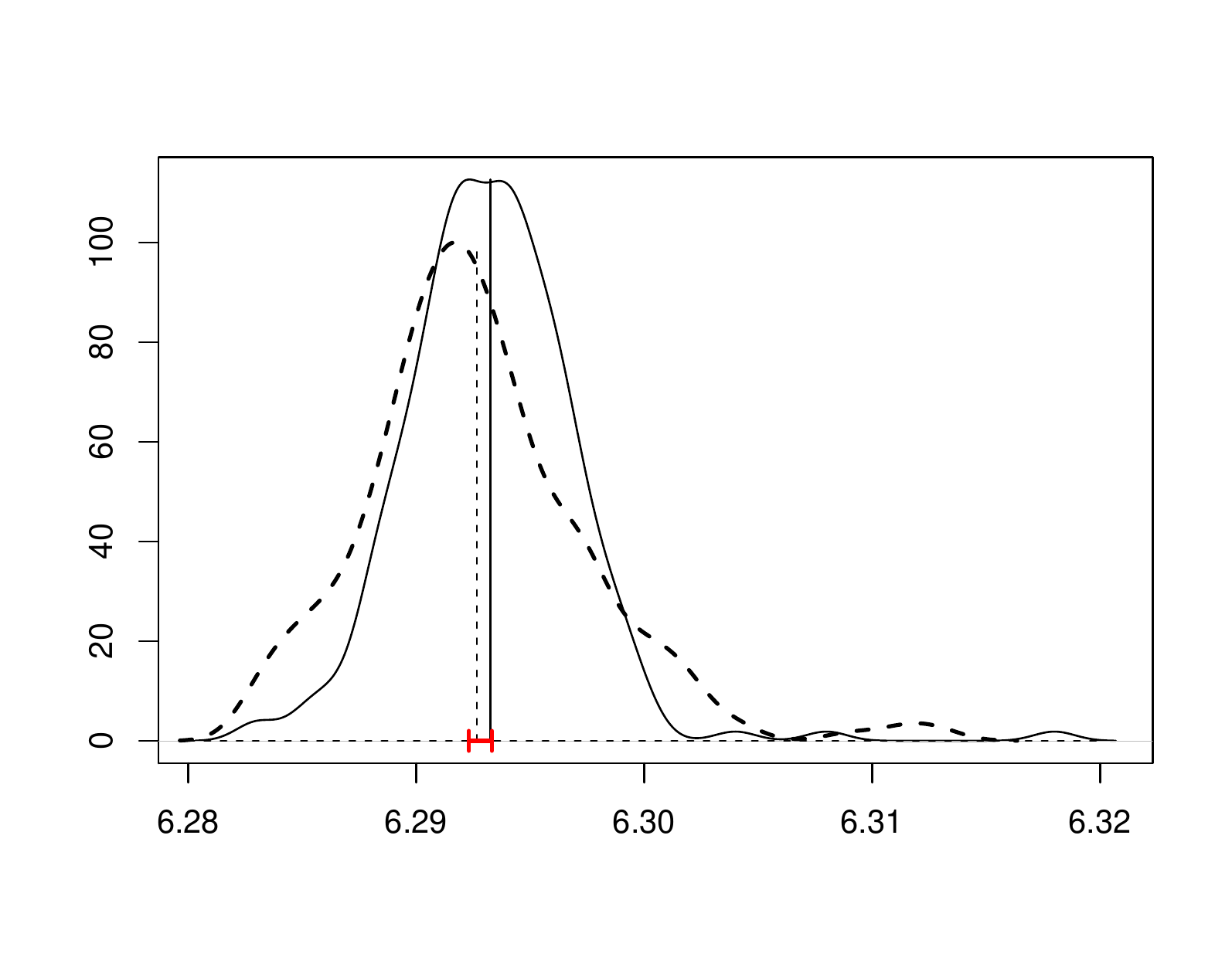}
	\end{center}
	\caption{Chip data: The fixed width confidence interval is marked in red. The vertical lines indicate the arithmetic means of both samples. Density estimates based on $245$ are added.}
	\label{ACFplots}
\end{figure}

\newpage

\section*{Acknowledgments}

A large part of the paper has been prepared during visits of A. Steland at Mejiro University, Tokyo. Both authors thank Nitis Mukhopadhyay for discussion, especially on two-stage procedures, at the International Symposium on Statistical Theory and Methodology for Large Complex Data 2018 held at Tsukuba University, and when he visited the Institute of Statistics at RWTH Aachen University. The authors gratefully acknowlegde the support of Takenori Takahashi, Mejiro University and Keio University Graduate School, and Akira Ogawa, Mejiro University, by providing the chip manufacturing data.



\end{document}